\documentclass[leqno,final]{siamltex}
\usepackage{amsmath}
\usepackage{graphicx}
\usepackage{epstopdf}
\usepackage{amssymb}
\usepackage{color}
\usepackage{epstopdf} 
\usepackage{multirow}
\usepackage{tabularx}

\newtheorem{remark}{Remark}

\newcommand{\cT}{\mathcal{T}}
\renewcommand{\div}{\mbox{\rm div\,}}

\newcommand{\mE}{\mathbb{E}}

\newcommand{\cF}{\mathcal{F}}

\newcommand{\mP}{\mathbb{P}}

\newcommand{\R}{\mathbf{R}}

\newcommand{\Ome}{\Omega}

\newcommand{\nab}{\nabla}

\newcommand{\bu}{\mathbf{u}}
\newcommand{\bP}{\mathbf{P}}
\newcommand{\bv}{\mathbf{v}}
\newcommand{\vf}{\mathbf{f}}
\newcommand{\be}{\mathbf{e}}
\newcommand{\beu}{\mathbf{e}_{\bu}}

\newcommand{\peta}{\pmb{\eta}}

\begin{document}

\title{Optimally Convergent Mixed Finite Element Methods for the Stochastic Stokes Equations}
\markboth{X. FENG, A. PROHL, AND L. VO}{MIXED AND STABILIZED METHODS FOR STOCHASTIC STOKES EQUATIONS}

\author{
Xiaobing Feng\thanks{Department of Mathematics, The University of Tennessee,
Knoxville, TN 37996, U.S.A. ({\tt xfeng@math.utk.edu})
The work of the this author was partially supported by
the NSF grant DMS-1620168.}
\and
Andreas Prohl\thanks{Mathematisches Institut, Universit\"at T\"ubingen, Auf der
	Morgenstelle 10, D-72076, T\"ubingen, Germany ({\tt prohl@na.uni-tuebingen.de}). }
\and
Liet Vo\thanks{Department of Mathematics, The University of Tennessee, Knoxville, TN 37996, U.S.A. ({\tt lvo6@vols.utk.edu})
The work of the this author was partially supported by the NSF grant DMS-1620168.}
}

\maketitle


\begin{abstract}
We propose some new mixed finite element methods for the time dependent stochastic 
Stokes equations with multiplicative noise, which use the Helmholtz decomposition
of the driving multiplicative noise. It is known \cite{LRS03} that  
	the pressure solution has a low regularity,
which manifests in sub-optimal convergence rates for well-known {\em inf-sup} 
stable mixed finite element methods in numerical simulations, see \cite{Feng}. 
We show that eliminating this gradient part from the noise in the numerical scheme 
leads to optimally convergent mixed finite element methods, and that this conceptional 
idea may be used to retool numerical methods that are well-known in the deterministic setting, 
including pressure stabilization methods, so that their optimal convergence 
properties can still be maintained in the stochastic setting. Computational experiments are also 
provided to validate the 
theoretical results and to illustrate the conceptional usefulness of the proposed numerical approach.
 
\end{abstract}




\section{Introduction}\label{sec-1}
 
This paper is concerned with fully discrete mixed finite element approximations of the following 
time-dependent stochastic Stokes equations with multiplicative noise for viscous incompressible 
fluids  covering the domain $D=(0,L)^d$ for $d=2,3$:
\begin{subequations}\label{eq1.1}
	\begin{alignat}{2} \label{eq1.1a}
	d{\bf u} &=\bigl[ \Delta {\bf u} -\nabla p + \vf\bigr] dt +{\bf B}({\bf u}) dW(t)  &&\qquad\mbox{ in}\, D_T:=(0,T)\times D,\\
	\div {\bf u} &=0 &&\qquad\mbox{ in}\, D_T,\label{eq1.1b}\\
	{\bf u}(0)&= {\bf u}_0 &&\qquad\mbox{ in}\, D,\label{eq1.1d}
	\end{alignat}
\end{subequations}
where ${\bf u}$ and $p$, respectively, denote the velocity field and the pressure of the fluid
which are spatially periodic with period $L>0$ in each coordinate direction. $\bu_0$ and 
$\mathbf{f}$ denote respectively the prescribed initial velocity and body force  which
are spatially periodic (see section \ref{sec-2} for the details).  For the sake of simplicity and ease of presentation,
 we assume $\{W(t); t\geq 0\}$ to be an ${\mathbb R}$-valued Wiener process; 
 see section \ref{sec-2} for further details.

When ${\bf B} \equiv {\bf 0}$, \eqref{eq1.1} is the well-known (deterministic) Stokes system;  
one motivation for studying \eqref{eq1.1a}--\eqref{eq1.1b} with ``random" force'' $\mathbf{f}+{\bf B}({\bf u})\frac{dW}{dt}$ is to develop mathematical models of this type for 
turbulent fluids \cite{Bensoussan95,HM06}. In addition to their importance in applied sciences and engineering, the Stokes equations are a  well-known PDE model with
saddle point structure, which requires special numerical discretizations to construct
optimally convergent methods;   
it should be noted that although the involved deterministic Stokes operator is
linear, system \eqref{eq1.1a}--\eqref{eq1.1b} is nonlinear due to the nonlinear function {${\bf B}$.}

The numerical analysis of the deterministic Stokes problem
 is well-established in the literature, see \cite{Brezzi_Fortin91,Girault_Raviart86,Rannacher}. Well-known numerical methods include 
{\em exactly divergence-free methods}, which approximate the velocity in exactly divergence-free
finite element spaces; {\em mixed finite element methods}, where the (discrete) {\em inf-sup} condition 
is the key criterion that distinguishes stable pairings of finite element ansatz spaces for the velocity 
(with more degrees of freedom) and the pressure (with less degrees of freedom); 
mixed methods allow a more flexible, broader application if compared to {\em exactly divergence-free methods}, thus putting them in the center of research on numerical methods for 
saddle point problems in the last decades. Another class of related numerical methods
are {\em stabilization methods} which were initiated in
\cite{hughes}, where the incompressibility constraint (\ref{eq1.1b}) is relaxed into
\begin{alignat}{2}\label{stabilize}
{\rm div}\, {\bf u} - \varepsilon \Delta p &= 0 &&\qquad \mbox{in } D_T, 
\end{alignat}
This relaxation allows for stable pairings of equal order (nodal-based) finite 
element ansatz spaces for both, 
velocity and pressure (putting $\varepsilon = {\mathcal O}(h^2)$, where $h>0$ is the spatial 
mesh size). We remark that optimal order error estimates had been obtained for 
all three classes of finite element methods in the deterministic setting (cf.~\cite{Girault_Raviart86,Brezzi_Fortin91}), where 
\begin{itemize}
\item {\em inf-sup} stable mixed finite element methods require the $H^1$-regularity of the 
pressure in order to optimally bound the best-approximation error for the pressure, which leads to 
optimal order convergence; cf.~\cite{Brezzi_Fortin91,Girault_Raviart86,Rannacher},
\item stabilization methods require the $H^1$-regularity of the pressure for convergence; cf.~\cite{hughes,Prohl}.
\end{itemize}

\smallskip
This work contributes to the numerical analysis of the stochastic Stokes problem \eqref{eq1.1}
(i.e., ${\bf B} \neq {\bf 0}$). By \cite{LRS03}, even for smooth datum functions $\bu_0$ and $\mathbf{f}$, the (temporal) regularity of
the pressure $p\in L^1\bigl(\Omega; W^{-1,\infty}(0,T;H^1(D)/\mathbb{R} )\bigr)$ 
is limited  in general due to the driving noise. In order to motivate its impact onto the pressure,
we here discuss the related question regarding
{\em $k$-independent} stability estimates  for the pair of
random variables $({\bf u}^{n+1}, p^{n+1})$ of the following time-implicit discretization of (\ref{eq1.1}) 
on a uniform mesh of $[0,T]$ with the mesh size $k>0$:   
\begin{subequations}\label{eq_new_intro} 
	\begin{alignat}{2}\label{eq_new_intro_a} 
	 {\bf u}^{n+1} - k \Delta {\bf u}^{n+1} + k \nabla p^{n+1}   &=  {\bf u}^n + k\vf^{n+1} +  {\bf B}({\bf u}^n) \Delta_{n+1} W &&\quad\mbox{in } D\, , \\
	 \div {\bf u}^{n+1} &= 0 &&\quad\mbox{in } D\, ,
	 \end{alignat}
	\end{subequations}
where $\Delta_{n+1} W :=W(t_{n+1})-W(t_{n})\thicksim {\mathcal N}(0,k)=\sqrt{k}{\mathcal N}(0, 1)$ and  $\vf^{n+1} = \vf(t_{n+1},\cdot) \in L^2(\Ome, L^2_{per}(D;\mathbb{R}^d))$.
A crucial observation for the motivation of this paper is that the pressure gradient on the 
left-hand side is scaled by $k$, while the noise term is of order $O(\sqrt{k})$.
Let us assume that the estimate  (\ref{eq3.4}) in Lemma \ref{lem3.1} 
for $\{\Delta {\bf u}^n\}_{n}$ taking values in $L^2(D; {\mathbb R}^d)$ is already shown, 
and we now look for a uniform bound for $\{\nabla p^n\}_n$ taking values in $L^2(D; {\mathbb R}^d)$.
The strategy for deriving such a stability estimate is to fix one $\omega \in \Omega$, and to multiply (\ref{eq_new_intro}) with $\nabla p^{n+1}(\omega)$: all the terms that involve the velocity vanish due to the incompressibility property and the periodic boundary condition, and we end up with
\begin{align}\label{eq_new_intro_0}
\frac{k}{2} \Vert \nabla p^{n+1}(\omega)\Vert^2
&\leq k\|\vf^{n+1}\|^2 +\Bigl( {\bf B}\bigl({\bf u}^n(\omega) \bigr) \Delta_{n+1} W(\omega), \nabla p^{n+1}(\omega)\Bigr)\, .
\end{align}
Note that the term on the right-hand side does not vanish since $\div {\bf B}({\bf u}^n)\neq 0$ 
for a general (Lipschitz) nonlinear mapping ${\bf B}$.
We now take expectations ${\mathbb E}[\cdot]$ on both sides, sum over all time steps, and use  
(\ref{eq3.4}), the facts that ${\bf B}({\bf u}^n)$ and $\Delta_{n+1} W$ are independent and 
${\mathbb E}\bigl[\vert \Delta_{n+1} W\vert^2\bigr] \leq Ck$,
and Young's inequality (with $\alpha >0$) to obtain the estimate
\begin{align*}\nonumber
&\frac{k}{2} \sum_{n=0}^{N-1}{\mathbb E}\bigl[\Vert \nabla p^{n+1} \Vert^2\bigr]
\leq  k\sum_{n=0}^{N-1} \mE\big[\|\vf^{n+1}\|^2\big] 
+ \sum_{n=0}^{N-1} {\mathbb E}\Bigl[
\Bigl( {\bf B} ({\bf u}^n   ) \Delta_{n+1} W , \nabla p^{n+1} \Bigr)\Bigr] \\ 
&\hskip 0.6in \leq  k\sum_{n=0}^{N-1} \mE\big[\|\vf^{n+1}\|^2\big]  
 +  \alpha k \sum_{n=0}^{N-1} {\mathbb E}\bigl[ \Vert {\bf B} ({\bf u}^n )\Vert^2\bigr] +
\frac{1}{4\alpha} \sum_{n=0}^{N-1}{\mathbb E}\bigl[\Vert \nabla p^{n+1} \Vert^2\bigr]\, .
\end{align*}
Taking $\alpha = \frac{1}{k}$ allows to absorb the last term on the right-hand side 
to the one on the left, but the remaining term is $\sum_{n=0}^{N-1} {\mathbb E}\bigl[ \Vert {\bf B} ({\bf u}^n )\Vert^2\bigr] \propto {\mathcal O}(k^{-1})$, therefore, we end up with the following $k$-dependent estimate:
\begin{equation}\label{eq_new_intro_1}
\frac{k}{4} \sum_{n=0}^{N-1}{\mathbb E}\bigl[\Vert \nabla p^{n+1}\Vert^2\bigr] \leq \frac{C}{k} + k\sum_{n=0}^{N-1} \mE\big[\|\vf^{n+1}\|^2\big]\,.
\end{equation}

The above consideration crucially affects the error analysis of a space-time discretization of \eqref{eq1.1a}--\eqref{eq1.1b}:
\begin{itemize}
\item  {\em Exactly divergence-free methods} require restricted settings of data, including the dimension, topology, and regularity of the spatial domain $D$. However, an optimal order 
error estimate can be proved for the velocity approximation, see \cite{CP12}, which 
uses the fact that no pressure is involved in the analysis.
\item The error estimate for the velocity approximation of {\em inf-sup} stable mixed finite element 
methods in \cite{Feng} was obtained based on   
a stability bound of type (\ref{eq_new_intro_1}) 
to bound the related best-approximation error for the pressure that appears
in (an auxiliary temporal discretization of) \eqref{eq1.1}, thus
leading to a sub-optimal error estimate for the velocity of
order ${\mathcal O}(k^{\frac12} +  h {k}^{-\frac12})$. 
The computational studies in \cite{Feng} suggest that this error  bound is sharp.
\end{itemize}
 
\smallskip
The first goal of the paper is to construct
optimally convergent {\em inf-sup} stable mixed finite 
element methods, with ``minimum" extra effort. Our main idea, which is partly borrowed from
	  \cite{CHP12}, is to perform the Helmholtz decomposition for the noise term
at each time step first, and then to determine the new velocity  
and pressure iterates simultaneously via the mixed finite element method. 
Below we shall use the semi-discrete time-stepping scheme (\ref{eq_new_intro}) to  motivate our strategy. Introducing the Helmholtz decomposition of $\mathbf{B}$ as follows
\begin{equation}\label{helmh1}
{\bf B}({\bf u}^n) = \nabla \xi^{n} + \pmb{\eta}^{n}\, \qquad \mbox{where} \quad
\div \pmb{\eta}^n = 0\, ,
\end{equation}
and setting $r^{n+1} :=   p^{n+1} - k^{-1} \Delta_{n+1} W \xi^{n}$,  then 
(\ref{eq_new_intro}) can be rewritten as
\begin{subequations}\label{eq_new_intro_ref}
	\begin{alignat}{2}\label{eq_new_intro_ref_a}
	 {\bf u}^{n+1} - k \Delta {\bf u}^{n+1} + k \nabla r^{n+1}   &=  {\bf u}^n + k \vf^{n+1}   +  \pmb{\eta}^n \Delta_{n+1} W   &&\qquad\mbox{in } D\, ,  \\
	 \div {\bf u}^{n+1} &= 0 &&\qquad \mbox{in } D\, .
	\end{alignat}
\end{subequations}
In contrast to estimate (\ref{eq_new_intro_1}) for $p^{n+1}$, it can be shown that the new pressure 
$r^{n+1}$ satisfies the following improved stability estimate (see Lemma~\ref{lem3.1}):
\begin{equation}\label{eq_new_intro_ref1}
k \sum_{n=0}^{N-1}{\mathbb E}\bigl[\Vert \nabla r^{n+1}\Vert^2\bigr] 
\leq  k\sum_{n=0}^{N-1} \mE\big[\|\vf^{n+1}\|^2\big]\,,
\end{equation}
which is a consequence of the divergence-free property of the modified noise 
term ({\em i.e.}, the last term on the right-hand side of  (\ref{eq_new_intro_ref_a})). 
Conceptually, this improved stability for the new  pressure $r^{n+1}$ is 
obtained by removing the {\em stochastic pressure} $\xi^n$ from the driving noise 
in (\ref{eq_new_intro_a}). As it will be detailed in Section \ref{sec-4}, any {\em inf-sup} 
stable mixed finite element discretization of (\ref{eq_new_intro_ref}) then gives optimally 
convergent velocity approximations (see Theorem \ref{thm4.5}), whose proof essentially 
relies on (\ref{eq_new_intro_ref1}). We also present optimal error estimates for 
(temporal averages of) the 
pressure approximations in $L^2$, which improve corresponding suboptimal 
estimates  in \cite{Feng}.

We therefore conclude by saying that it is essential to identify the proper role of the 
semi-discrete pressures, namely, $\{p^n\}_n$ in (\ref{eq_new_intro}) vs.~$\{r^n\}$ in (\ref{eq_new_intro_ref}), for {\em inf-sup} stable mixed finite element methods
for \eqref{eq1.1} in order to construct optimally convergent mixed methods. Moreover, 
this insight  also suggests how to construct optimally convergent stabilization 
methods for \eqref{eq1.1} which circumvent the {\em inf-sup} stability criterion 
for mixed element methods, and hence allow a more efficient discretization such as
\begin{subequations}\label{eq_new_intro_ref_2}
	\begin{alignat}{2}\label{eq_new_intro_ref_2a}
	 {\bf u}_{\varepsilon}^{n+1} - k \Delta {\bf u}_{\varepsilon}^{n+1} + k \nabla r_{\varepsilon}^{n+1}  
	 &=  {\bf u}_{\varepsilon}^n + k \vf^{n+1}   +  \pmb{\eta}^n_\varepsilon \Delta_{n+1} W 
	 &&\qquad \mbox{in } D\, ,  \\
	 \div {\bf u}_{\varepsilon}^{n+1}- \varepsilon \Delta r^{n+1}_\varepsilon &= 0
	 &&\qquad\mbox{in } D\, , \label{eq_new_intro_ref_2b} 
	 \end{alignat}
\end{subequations}
for which $\varepsilon = {\mathcal O}(h^2)$ will be shown to be the optimal choice in section \ref{sec-5}.
The error analysis in section \ref{sec-5} verifies optimal order convergence 
for a standard finite element discretization of (\ref{eq_new_intro_ref_2}) which employs 
the same finite element space for approximating both, ${\bf u}_{\varepsilon}^{n+1}$ and $r^{n+1}_\varepsilon$; see~Theorem \ref{error-thm}. Corresponding computational 
studies in section \ref{sec-6} support the conclusion that the choice of pressure in the stabilization
is crucial for achieving an optimally convergent {\em stabilization method} for \eqref{eq1.1}.

The remainder of this paper is organized as follows. In section \ref{sec-2}, we 
give exact assumptions on the data in \eqref{eq1.1}, and
recall the definition and known properties of the (strong) variational 
solution for problem \eqref{eq1.1}.
In sections \ref{sec-3} and \ref{sec-4}, we analyze the {\em Helmholtz decomposition 
	enhanced} Euler-Maruyama time-stepping scheme \eqref{helmh1}--\eqref{eq_new_intro_ref} 
    and its mixed finite element approximations, and establish   
    the optimal convergence for both. Section \ref{sec-5} establishes 
    optimal convergence for the stabilized scheme \eqref{eq_new_intro_ref_2} and 
    its equal-order finite element approximations. 
 Two-dimensional numerical experiments and computational studies are given in section \ref{sec-6}   to validate the theoretical error bounds, and 
 to computationally evidence that a proper selection of the pressure for the construction
 of optimally convergent mixed methods is indeed necessary.

\section{Preliminaries}\label{sec-2}
\subsection{Notations}\label{sec-2.1}
Standard function and space notation will be adopted in this paper. 
For example, $H^\ell_{per}(D, {\mathbb R}^d)\, (\ell\geq 0)$ denotes the subspace of 
the Sobolev space $H^\ell(D, {\mathbb R}^d)$ consisting of ${\mathbb R}^d$-valued 
periodic functions with period $L$ in each spatial coordinate direction, and  $(\cdot,\cdot):=(\cdot,\cdot)_D$ denote the standard $L^2$-inner product, 
with induced norm $\Vert \cdot \Vert$.
Let $(\Omega,\cF, \{\cF_t\},\mP)$ be a filtered probability space with the 
probability measure $\mP$, the 
$\sigma$-algebra $\cF$ and the continuous  filtration $\{\cF_t\} \subset \cF$. 
For a random variable $v$ defined on $(\Omega,\cF, \{\cF_t\},\mP)$,
let ${\mathbb E}[v]$ denote the expected value of $v$. 
For a vector space $X$ with norm $\|\cdot\|_{X}$,  and $1 \leq p < \infty$, we define the Bochner space
$\bigl(L^p(\Omega,X); \|v\|_{L^p(\Omega,X)} \bigr)$, where
$\|v\|_{L^p(\Omega,X)}:=\bigl({\mathbb E} [ \Vert v \Vert_X^p]\bigr)^{\frac1p}$.
Throughout this paper, unless it is stated otherwise,
	 we shall use $C$ to denote a generic positive constant
	which may depend on $T$, the datum functions $\bu_0$ and $\vf$, and the domain $D$ 
	but is independent of the mesh parameter $h$ and $k$.
 
We also define 
\begin{align*}
{\mathbb H}&:= \bigl\{{\bf v}\in  L^2_{per}(D; {\mathbb R}^d) ;\,\div \bv=0 \mbox{ in }D \bigr\}\, , \\
{\mathbb V}&:=\bigl\{{\bf v}\in  H^1_{per}(D; {\mathbb R}^d) ;\,\div \bv=0 \mbox{ in }D \bigr\}\, .
\end{align*}

We recall from \cite{Girault_Raviart86} that the (orthogonal) Helmholtz projection 
${\bf P}_{{\mathbb H} }: L^2_{per}(D; {\mathbb R}^d)$ $\rightarrow {\mathbb H} $ is defined 
by ${\bf P}_{{\mathbb H}} {\bf v} = \pmb{\eta}$ for every ${\bf v} \in L^2_{per}(D; {\mathbb R}^d)$, 
where $(\pmb{\eta}, \xi) \in {\mathbb H} \times H^1_{per}(D)/\mathbb{R}$ is a unique tuple such that 
$${\bf v} = \pmb{\eta} + \nabla \xi\, , $$
and   $\xi\in H^1_{per}(D)/\mathbb{R}$  solves the following Poisson problem 
(cf. \cite{Babutzka18}):
\begin{equation}\label{poisson}
	(\nabla \xi, \nabla q) = ({\bf v},\nabla q)  \qquad \forall\, q \in H^1_{per}(D)\, .
\end{equation}
In this paper we denote by ${\bf A} := {\bf P}_{{\mathbb H}} \Delta:  H^2(D; {\mathbb R}^d) \rightarrow {\mathbb H}$ the Stokes operator. 

We assume that ${\bf B}: L^2(\Ome;H^1_{per}(D; {\mathbb R}^d)) \rightarrow L^2(\Ome; H^1_{per}(D; {\mathbb R}^d))$ is Lipschitz continuous and has linear growth, {\em i.e.}, 
there exists a constant $C  > 0$ such that for all ${\bf v}, {\bf w} \in L^2_{per}(D; {\mathbb R}^d)$,
 \begin{subequations}\label{eq2.6}
	\begin{align}\label{eq2.6a}
	\|{\bf B}({\bf v})-{\bf B}({\bf w})\|  &\leq C \|{\bf v}-{\bf w}\|\, , \\
	\|{\bf B}({\bf v})\|  &\leq C  \bigl(1+ \|{\bf v}\| \bigr)\, ,   \label{eq2.6b}\\
	{\|\mathcal{D} \mathbf{B}\|_*} &\leq C \, ,\label{eq2.6c}
	\end{align}
where $\mathcal{D}\mathbf{B}$ denotes the Gateaux derivative of $\mathbf{B}$, and $\|\cdot\|_*$ is 
	its operator norm. 
\end{subequations}
  
\subsection{Variational formulation of the stochastic Stokes equations}\label{sec-2.2}
We first recall the solution concept for \eqref{eq1.1}, and refer to \cite{Chow07,PZ92} 
for its existence and uniqueness.
 
\begin{definition}\label{def2.1} 
	Given $(\Omega,\cF, \{\cF_t\},\mP)$, let $W$ be an ${\mathbb R}$-valued Wiener 
	process on it. 	Suppose ${\bf u}_0\in L^2(\Omega, {\mathbb V})$ and $\vf \in L^2(\Ome;L^2((0,T);L^2_{per}(D;\mathbb{R}^d)))$.
	An $\{\cF_t\}$-adapted stochastic process  $\{{\bf u}(t) ; 0\leq t\leq T\}$ is called
	a variational solution of \eqref{eq1.1} if ${\bf u} \in  L^2\bigl(\Omega; C([0,T]; {\mathbb V})) 
	\cap L^2\bigl(0,T;H^2_{per}(D;{\mathbb R}^d)\bigr)$,
and satisfies $\mP$-a.s.~for all $t\in (0,T]$
\begin{align}\label{eq2.8a}
		\bigl({\bf u}(t),  {\bf v} \bigr) + \int_0^t  \bigl(\nab {\bf u}(s), \nab {\bf v} \bigr) 
		\,  ds&=({\bf u}_0, {\bf v}) +\int_0^t\big(\vf(s),{\bf v}\big)\, ds \\\nonumber
		&\qquad +{\int_0^t  \Bigl( {\bf B}\bigl({\bf u}(s)\bigr), {\bf v} \Bigr)\, dW(s)}  \quad\forall  \, {\bf v}\in {\mathbb V}\, . 
\end{align}

\end{definition}

\smallskip
The following estimates from \cite{CHP12, Feng} establish the H\"older continuity 
in time of the variational solution in various spatial norms.
\begin{theorem}\label{thm2.2}
	Additionally suppose ${\bf u}_0 \in L^2\bigl(\Omega; {\mathbb V} \cap H^2_{per}(D; {\mathbb R}^d)\bigr)$ and $\vf \in L^2(\Ome,C^{\frac12}([0,T]);H^1_{per}(D;\mathbb{R}))$. There exist a  constant $C>0$, such that the variational solution to problem \eqref{eq1.1} satisfies
	 for $s,t \in [0,T]$
	\begin{subequations}
		\begin{align}\label{eq2.20a}
		& {\mathbb E}\bigl[\|{\bf u}(t)-{\bf u}(s)\|^2 \bigr]+ {\mathbb E}\Bigl[\int_s^t \|\nabla\bigl({\bf u}(\tau)-{\bf u}(s)\bigr)\| ^2 \, d\tau \Bigr]
		\leq C|t-s|\, ,\\ \label{eq2.20b}
		& {\mathbb E}\bigl[\|\nabla \bigl({\bf u}(t)-{\bf u}(s)\bigr)\|^2 \bigr]+ {\mathbb E}\Bigl[\int_s^t \|{\bf A}\bigl({\bf u}(\tau)-{\bf u}(s)\bigr)\|^2\, d\tau\Bigr] 
		\leq C|t-s|\, .
		\end{align}
	\end{subequations}
\end{theorem}

\begin{remark}
	To avoid the technicality of tracking the required ``minimum" 
	assumptions on $\bu_0$ and $\vf$ for each stability and/or error estimate, unless it
	is stated otherwise, we shall implicitly make the ``maximum" assumption ${\bf u}_0 \in L^2\bigl(\Omega; {\mathbb V} \cap H^2_{per}(D; {\mathbb R}^d)\bigr)$ and $\vf \in L^2(\Ome,C^{\frac12}([0,T]);H^1_{per}(D;\mathbb{R}))$ in the rest of the paper.
\end{remark}	

\subsection{Definition and role of the pressure}\label{sec-2.3}
The Definition \ref{def2.1} only addresses the velocity $\mathbf{u}$ in the stochastic PDE \eqref{eq1.1}; a corresponding pressure which satisfies a proper formulation (see Theorem \ref{thm 2.2} below) may be constructed after the existence of a velocity field ${\bf u}$ has been established.
We therefore consider processes
\[
\mathbf{U}(t):=\int_0^t \mathbf{u}(s)\, ds\,,\quad 
 \mathbf{F}(t):=\int_0^t \mathbf{f}(s)\, ds \quad  \mbox{and} \quad\mathbf{G}(t):= \int_0^t  {\bf B}\bigl({\bf u}(s))\, dW(s)\, .
\]
Evidently, $\mathbf{U} \in L^2\bigl(\Omega, L^2(0,T; H^2_{per}(D,\mathbb{R}^d))\bigr)$ and  
$\mathbf{G}\in L^2\bigl(\Omega, L^2(0,T; L^2_{per}(D,\mathbb{R}^d))\bigr)$, and
 \eqref{eq2.8a} therefore implies 
\begin{equation}\label{eq2.8c}
\bigl( \mathbf{u}(t)-  \Delta \mathbf{U}(t)-\mathbf{u}_0 - \mathbf{F}(t) -\mathbf{G}(t),   \mathbf{v} \bigr) =0
 \qquad\forall  \, {\bf v}\in {\mathbb V},\,  t\in (0,T), \, {\mathbb P}\mbox{-a.s.}
\end{equation} 
By the Helmholtz decomposition
 \cite[Theorem 4.1 and Remark 4.3]{LRS03},
there exists a unique  $P \in  {L^2\bigl(\Omega, L^2(0,T; H^1_{per}(D))/\mathbb{R}\bigr)}$  
such that  
\begin{equation}\label{eq2.8d}
\nabla P(t)= - \bigl[ \mathbf{u}(t)- \Delta \mathbf{U}(t)-\mathbf{u}_0   -\mathbf{F}(t) 
- \mathbf{G}(t) \bigr]
\qquad\forall  \, t\in (0,T),\, {\mathbb P}\mbox{-a.s.} 
\end{equation}                                                                             in the distributional sense. It is shown in \cite[Section 5]{LRS03}, that 
its distributional time derivative 
$p := \partial_t P \in L^1\bigl(\Omega; W^{-1,\infty}(0,T; H^1_{per}(D)/\mathbb{R})\bigr)$.
As a consequence, we have the following result.
                                                                                 
\begin{theorem}\label{thm 2.2}
	 Let $\{{\bf u}(t) ; 0\leq t\leq T\}$ be the variational solution of \eqref{eq1.1}. There exists a unique adapted process  
	 $P\in {L^2\bigl(\Omega, L^2(0,T; H^1_{per}(D)/\mathbb{R}\bigr)}$ such that $(\mathbf{u}, P)$ satisfies  $\mP$-a.s.~for all $t\in (0,T]$
\begin{subequations}\label{eq2.10}
	\begin{align}\label{eq2.10a}
	&\bigl({\bf u}(t),  {\bf v} \bigr) + \int_0^t  \bigl(\nab {\bf u}(s), \nab {\bf v} \bigr) \, ds - \bigl(  \div \mathbf{v}, P(t) \bigr)
	\\\nonumber
	&\,\, =({\bf u}_0, {\bf v}) + \int_0^t \big(\vf(s), \mathbf{v}\big)\, ds  
	 +  {\int_0^t  \Bigl( {\bf B}\bigl({\bf u}(s)\bigr), {\bf v} \Bigr)\, dW(s)}  \quad\forall  \, {\bf v}\in H^1_{per}(D; \mathbb{R}^d)\, , \nonumber \\ 
    &\bigl(\div {\bf u}, q \bigr) =0 \qquad\forall \, q\in  L^2_{per}(D)\,  .  \label{eq2.10b}
		\end{align}
\end{subequations}
\end{theorem}

\smallskip
System \eqref{eq2.10} can be regarded as a mixed formulation for the stochastic Stokes system 
\eqref{eq1.1}, where the (time-averaged) pressure $P$ is defined.
Below, we also define another time-averaged ``pressure" $$R(t) := P(t) - \int_0^t\xi(s)\, dW(s),$$  where we
use the Helmholtz decomposition ${\bf B}(\bu(t)) = \pmb{\eta}(t) + \nabla \xi(t)$, where 
$\xi\in H^1_{per}(D)/\mathbb{R}$  ${\mathbb P}\mbox{-a.s.}$ such that 
\begin{equation}\label{eq2.8f} 
\bigl(\nabla \xi(t), \nabla \phi \bigr) =  \bigl( {\bf B}(\bu(t)) , \nabla \phi 
\bigr)\qquad \forall\, \phi \in H^1_{per}(D)\, .
\end{equation}
Then, (\ref{eq2.8d}) can be rewritten as 
\begin{equation}\label{eq2.8e}
\nabla R(t)= - \Bigl[ \mathbf{u}(t)- \Delta \mathbf{U}(t)-\mathbf{u}_0 - \mathbf{F}(t) -\int_0^t \pmb{\eta}(s)\, dW(s) \Bigr]
\qquad\forall  \, t\in (0,T),\, {\mathbb P}\mbox{-a.s.} 
\end{equation}
The time averaged ``pressure" $\{R(t); 0\leq t\leq T\}$ will also be a target process to be approximated in our numerical methods. 

\section{Semi-discretization in time}\label{sec-3}
In this section we study the stability and convergence properties of a 
Helmholtz decomposition enhanced Euler-Maruyama time discretization scheme 
that is based on (\ref{eq_new_intro_ref}), where the stochastic pressure is removed 
from the noise term via the Helmholtz decomposition; but its ${\mathbb V}$-valued 
velocity approximation  $\{ {\bf u}^{n+1}\}_n$ still solves the original Euler-Maruyama 
scheme (\ref{eq_new_intro}).

\subsection{Formulation of the time-stepping scheme} \label{sec-3.1}
In the following, let $N$ be a positive integer, $k = \frac{T}{N}$,  and $t_n= nk$ for 
	$n  = 0, 1, \ldots, N$ be  a uniform mesh that covers $[0,T]$.   
 
\medskip
\noindent
{\bf Algorithm 1}

Let ${\bf u}^0={\bf u}_0$.  For $n=0,1,\ldots, N-1$ do the following steps:

\smallskip
{\em Step 1:}    Find $\xi^{n} \in  L^2\bigl(\Omega,H^1_{per}(D)/\mathbb{R}\bigr)$ by solving 
\begin{equation}\label{eq_WF_xi1a}
\big({\nab\xi^{n}}, \nab\phi\big) = \big({\bf B}({\bf u}^n),\nab\phi\big) 
	\qquad \forall \, \phi\in H^1_{per}(D)\, .
\end{equation}

\smallskip
{\em Step 2:} Set $\pmb{\eta}^n := {\bf B}({\bf u}^n)-\nabla \xi^{n}$, and find
$({\bf u}^{n+1},r^{n+1}) \in L^2\bigl(\Omega, {\mathbb V} \times
L^2_{per}(D)/\mathbb{R}\bigr)$ by solving 
\begin{subequations}\label{eq_new}
	\begin{align} \label{eq_newa} 
	\big({\bf u}^{n+1},{\bf v}\big) + & k \big(\nabla {\bf u}^{n+1}, \nabla {\bf v}\big)- k \big(\div {\bf v}, r^{n+1} \big) \\
	&\quad  = \big({\bf u}^n,{\bf v}\big)  + k\big(\vf^{n+1},\mathbf{v}\big) +  \big( \pmb{\eta}^n \Delta_{n+1} W  ,{\bf v}\big) 
	\quad \forall \, {\bf v} \in H^1_{per}(D, {\mathbb R}^d),  \nonumber\\
	\big(\div {\bf u}^{n+1},q&\big) = 0 \qquad\forall \, q\in L^2_{per}(D)\,.\label{eq_newb}
	\end{align}
\end{subequations}

\smallskip
{\em Step 3:} {Define $p^{n+1} := r^{n+1} + k^{-1} \xi^n \Delta_{n+1} W $.} 

\medskip
 \begin{remark}\label{remy1}
By the elliptic regularity theory, see ~\cite[p.~13]{Girault_Raviart86}, the solution of
(\ref{eq_WF_xi1a}) is in $\xi^{n} \in L^2\bigl(\Omega,H^2_{per}(D)/\mathbb{R}\bigr)$, and satisfies {\rm Lebesgue}-a.e.
	\begin{subequations}\label{rem-1}
	\begin{alignat}{2}\label{eq_new_intro_ref_a2}
	 -\Delta \xi^n   &=  -{\rm div}\, {\bf B}({\bf u}^n)   &&\qquad\mbox{in } D\, . 
	\end{alignat}
\end{subequations}
Moreover there exists a constant $C>0$ such that 
\begin{equation}\label{rem-1a}
\Vert \xi^n\Vert_{H^2/\mathbb{R}} \leq C \, \Vert {\rm div}\, {\bf B}({\bf u}^n) \Vert\, . 
\end{equation}
\end{remark}

The solvability  of Algorithm 1 is clear because  
a linear coercive elliptic PDE problem is solved at each step. 
{\em Step 1} in Algorithm 1 requires to solve a Poisson problem  
(\ref{eq_WF_xi1a}), which only slightly increases 
the computational cost if a fast solver is used to solve them.  The iterates $\{(\mathbf{u}_n, r_n)\}_n$ and $\{p_n\}_n$ 
defined in {\em Step 2} and {\em 3} aim to approximate $\{ (\mathbf{u}(t),r(t)); 0\leq t\leq T\}$ and 
$\{ p(t); 0\leq t\leq T\}$, respectively. See subsection \ref{sec-3.4} for details. 

\subsection{Stability estimates} \label{sec-3.2}
In this subsection we present some stability estimates for the time-stepping scheme given in Algorithm 1. 
All these estimates, in particular the estimate for $\{ \nabla r^{n+1}\}_n$,  will play an important 
role in establishing  optimal order error estimates for the fully mixed finite element 
discretization to be given in the next section.

\begin{lemma}\label{lem3.1}
Let $\{ ({\bf u}^{n+1}, r^{n+1})\}_n$ be generated by {\rm Algorithm 1}. There exists
a constant $C>0 $, 
such that 
\begin{eqnarray}
\label{eq3.4}
	&&\quad \max_{1\leq n\leq N} {\mathbb E}\bigl[\|\nabla {\bf u}^n\|^2 \bigr] +{\mathbb E}\bigl[\sum^N_{n=1}\|\nabla ({\bf u}^n-{\bf u}^{n-1})\|^2 \Bigr]  +{\mathbb E} \bigl[k\sum^N_{n=1}\|{\bf A} {\bf u}^n\|^2 \bigr] 
	 \leq C\,,  \\ \label{eq3.4a}
	 &&\quad {\mathbb E}\bigl[k\sum_{n=1}^N \|\nabla r^n\|^2 \bigr]
	 \leq	C\, .  	
\end{eqnarray}

\end{lemma}

\begin{proof} 
	See \cite{CP12,Feng} for a proof of estimate (\ref{eq3.4});
estimate (\ref{eq3.4a}) was already proved in section \ref{sec-1}  after 
(\ref{eq_new_intro_ref}) was introduced. 
 We note that the periodicity of $\mathbf{B}(\bu^n)$
was crucially used in the proof of (\ref{eq3.4}) to avoid the boundary integral terms 
 arising from integration by parts in the noise term.
\end{proof}

\subsection{Error estimate for the velocity approximation} \label{sec-3.3}
Since the velocity approximation  $\{ {\bf u}^{n+1}\}_n$ generated by Algorithm 1 
also solves the original Euler-Maruyama time-stepping scheme (\ref{eq_new_intro}), the following 
optimal order error estimate for $\{ {\bf u}^{n+1}\}_n$ was established  in \cite{CP12,Feng}.
 
\begin{theorem}\label{thm3.3}
Let $\{ ({\bf u}^{n+1}, r^{n+1})\}_n$ be generated by {\rm Algorithm 1}. There exists
a constant $C>0$, 
such that 
	\begin{equation}\label{eq3.6}
	\max_{1\leq n\leq N} \left({\mathbb E}\bigl[ \|{\bf u}(t_n)-{\bf u}^n\|^2\bigr]\right)^{\frac12}
	+ \left({\mathbb E}\bigg[k\sum_{n=1}^{N} \|\nabla \bigl({\bf u}(t_n)-{\bf u}^n\bigr)\|^2 \bigg]\right)^{\frac12} \leq C k^{\frac12}.
	\end{equation}

\end{theorem}

We note that the proof of the above error estimate crucially uses the fact that 
$\mathbf{u}^n$ is exactly divergence-free for each $0\leq n\leq N$.

\subsection{Error estimates for the pressure approximations}\label{sec-3.4}
An optimal order error estimate was obtained in \cite{Feng} 
for $\{P(t_n)\}_n$ via the Euler-Maruyama 
time-stepping scheme \eqref{eq_new_intro}. 
For the reader's convenience, we here give its proof. 

\begin{theorem}\label{thm3.4a}
	Let $\{p^n; 1\leq n\leq N\}$ be the  pressure in (\ref{eq_new_intro}),   
		and $\{P(t); 0\leq t\leq T\}$ be defined in Theorem \ref{thm 2.2}. 
		There exists a constant $C>0$, 
		such that 
	\begin{align}\label{eq3.26a}
	\Bigl(\mathbb{E}\bigl[ \|P(t_m)-k\sum^m_{n=1}p^n  \|^2 \bigr] \Bigr)^{\frac12}
	\leq C\,  k^{\frac12}\, ,
	\qquad m=1,2,\cdots, N\,.
	\end{align}
\end{theorem}

\begin{proof}
Consider (\ref{eq_new_intro_a}), and take the sum over steps $0 \leq n \leq m-1$. We denote ${\bf U}^{m} := k \sum_{n=0}^{m-1} {\bf u}^{n+1}$ and $P^m := 
k \sum_{n=0}^{m-1} p^{n+1}$, and therefore obtain
\begin{equation}\label{disc-1} {\bf u}^{m} - {\bf u}^0 -  \Delta {\bf U}^{m} + \nabla P^m = k\sum_{n=0}^{m-1} \vf^{n+1} + \sum_{n=0}^{m-1} {\bf B}({\bf u}^n)\Delta_{n+1} W\, .
\end{equation}
We subtract this equation  from (\ref{eq2.8d}) at time $t = t_m$, and denote
${\bf E}_{\bf U}^m := {\bf U}(t_m) - {\bf U}^m\in L^2(\Ome; H^1_{per}(D;\mathbb{R}^d))$, 
and $E_{P}^m := P(t_m) - P^m \in L^2(\Ome; L^2_{per}(D))$.
By the stability of the divergence operator, there exists $\beta>0$, such that 
\begin{eqnarray}  \label{pressure-e}
\frac{1}{\beta} \Vert E_{P}^m\Vert &\leq&  \sup_{{\bf v} \in H^1_{per}(D; {\mathbb R}^d)}\frac{(E_{P}^m, {\rm div}\, {\bf v})}{\Vert \nabla {\bf v} \Vert} \\ \nonumber
&\leq& \Vert {\bf u}(t_m) - {\bf u}^m\Vert + \Vert \nabla {\bf E}^m_{\bf U}\Vert
+ \Bigl\Vert \sum_{n=0}^{m-1} \bigl( {\bf B}({\bf u}(t_{n})\bigr) - {\bf B}({\bf u}^n)\bigr) \Delta_{n+1} W\Bigr\Vert \\ \nonumber 
&& + \Bigl\Vert  
\sum_{n=0}^{m-1} \int_{t_n}^{t_{n+1}} {\bf B}\bigl({\bf u}(s) \bigr) - {\bf B}\bigl({\bf u}(t_{n})\bigr) \, dW(s)\Bigr\Vert \\\nonumber
&& + \bigg\|\sum_{n=0}^{m-1}\int_{t_n}^{t_{n+1}}\big(\vf(s)  - \vf(t_{n+1})\big)\,ds\bigg\| 
=: {\tt I} +\ldots+{\tt V}\, .
\end{eqnarray}
Taking squares on both sides, and then applying expectations, Theorem \ref{thm3.3} in combination with H\"older's inequality leads to
$$\frac{1}{\beta^2}{\mathbb E}[\Vert E_{P}^m\Vert^2] \leq C k + 
{\mathbb E}\bigl[\Vert {\tt III}\Vert^2 \bigr] + {\mathbb E}\bigl[\Vert {\tt IV}\Vert^2 \bigr] + \mE\big[\|{\tt V}\|^2\big]\, .$$
By Ito's isometry, and (\ref{eq2.6a}), as well as (\ref{eq2.20a}), and Theorem \ref{thm3.3}, we find the bounds
$${\mathbb E}\bigl[\Vert {\tt III}\Vert^2 + \Vert {\tt IV}\Vert^2\bigr]   
\leq C{\mathbb E}\bigl[ k \sum_{n=0}^{m-1} \Vert {\bf u}(t_{n}) - {\bf u}^n\Vert^2\bigr] + C k
\leq Ck\,, $$
and by using Cauchy-Schwarz inequality, we get
\begin{align*}
\mE\big[\|{\tt V}\|^2\big] 
&\leq \mE\bigg[\sum_{n=0}^{m-1}\int_{t_n}^{t_{n+1}}\|\vf(s) - \vf(t_{n+1})\|^2\,ds\bigg] \leq Ck,
\end{align*}
which lead to the desired estimate \eqref{eq3.26a}. 
\end{proof}

We now consider the pressure $R^m := k \sum_{n=0}^{m-1} r^{n+1}$, where $\{r^n\}_{n}$ 
is defined by Algorithm 1. Using the new notation (\ref{disc-1}) can be written as 
\begin{equation}\label{disc-2}{\bf u}^{m} - {\bf u}^0 -  \Delta {\bf U}^{m} + \nabla R^m = k\sum_{n=0}^{m-1} \vf^{n+1} + \sum_{n=0}^{m-1} \bigl( {\bf B}({\bf u}^n) - \nabla \xi^n\bigr)\Delta_{n+1} W\, .
\end{equation}
We again subtract this equation from (\ref{eq2.8e}) at time $t = t_m$, and adapt the error notation in (\ref{pressure-e}),
\begin{equation}\nonumber \frac{1}{\beta} \Vert E_{R}^m\Vert \leq  \sup_{{\bf v} \in H^1_{per}(D; {\mathbb R]^d)}}\frac{(E_{R}^m, {\rm div}\, {\bf v})}{\Vert \nabla {\bf v} \Vert} \\ \nonumber
\leq  {\tt I} +\ldots+{\tt IV} + {\tt V} + {\tt VI}\, ,
\end{equation}
where {\tt V} is the same as above, hence, ${\mathbb E}\bigl[ \Vert {\tt V}\Vert^2\bigr] \leq Ck$, and
$${\tt VI} := \Bigl\Vert \sum_{n=0}^{m-1} \nabla \bigl(\xi(t_n) - \xi^n\bigr) \Delta_{n+1}W \Bigr\Vert +\Bigl\Vert  
\sum_{n=0}^{m-1} \int_{t_n}^{t_{n+1}} \nabla \bigl( \xi(s)   - \nabla \xi(t_{n})\bigr) \, dW(s)\Bigr\Vert =: {\tt VI}_{\tt 1} + {\tt VI}_{\tt 2}
\, .$$
By a stability result for the Poisson problems (\ref{eq2.8f}), (\ref{eq_WF_xi1a}), and property (\ref{eq2.6a}), we 
easily obtain, thanks to Theorem \ref{thm3.3},
$${\mathbb E}\bigl[ \Vert {\tt VI}_{\tt 1}\Vert^2\bigr] 
\leq C{\mathbb E} \biggl[ k \sum_{n=0}^{m-1} \Vert {\bf u}(t_{n}) - {\bf u}^n\Vert^2 \biggr] 
\leq Ck\, .$$
Similarly, we get ${\mathbb E}\bigl[ \Vert {\tt VI}_{\tt 2}\Vert^2\bigr] \leq Ck$.
We collect this result below.
 
\begin{corollary}\label{thm3.4}
	Let $\{r^n; 1\leq n\leq N\}$ be the discrete process from Algorithm 1.   
	There exists a constant $C>0$ 
	such that 
	\begin{align}\label{eq3.26}
	\biggl(\mathbb{E}\bigl[ \|R(t_m)-k\sum^m_{n=1}r^n  \|^2 \bigr] \biggr)^{\frac12}
	\leq C  k^{\frac12}\,, \qquad \qquad m=1,2,\cdots, N\,. 
	\end{align}

\end{corollary}

\section{Fully discrete, {\em inf-sup} stable mixed finite element method}\label{sec-4}
In this section, we discretize  Algorithm 1 
in space via an {\em inf-sup} stable mixed finite element method. We choose the prototypical
Taylor-Hood  mixed finite element (see, {\em e.g.}, \cite{Girault_Raviart86,Brezzi_Fortin91}) 
as an example and give a detailed error
analysis for the resulted fully discrete method, but we remark 
that the convergence analysis  below also applies to general {\em inf-sup} stable 
mixed finite elements.  

\subsection{Preliminaries}\label{sec-4.1}
Let $\cT_h$ be a quasi-uniform triangular or rectangular mesh of $D\subset \R^d$ with mesh size $0<h <<1$. 
We define the following finite element spaces:
\begin{align*}
{\mathbb X}_{h} &=\bigl\{{\bf v}_h\in H^1_{per}( {D}; {\mathbb R}^d);\, {\bf v}_h|_K\in P_2(K, {\mathbb R}^d)\,\,\forall \, K\in  \cT_h\bigr\}\, ,\\
W_{h} &=\bigl\{q_h\in H^1_{per}(D)/\mathbb{R};\, q_h|_K\in P_1(K)\,\,\forall \, K\in \cT_h\bigr\}\, ,\\
S_{h} &=\bigl\{\phi_h\in  H^1_{per}(D)/\mathbb{R};\, \phi_h|_K\in P_\ell(K)\,\,\forall \, K\in \cT_h\bigr\}\, ,
\end{align*}
where $P_\ell(K; {\mathbb R}^d)$ ($\ell\geq 1$) denotes the set of ${\mathbb R}^d$-valued polynomials of degree less than or equal to $\ell$
over the element $K\in \cT_h$. In general, we require that $S_{h}\subseteq W_{h}$, 
in particular, we choose $\ell=1$ so that $S_{h}=W_{h}$ in this section.

We recall that the pair $(\mathbb{X}_{h}, {W}_{h})$ satisfies the (discrete)
{\em inf-sup} condition: there
exists an $h$-independent constant $\gamma>0$ such that
\begin{equation}\label{inf-sup-TH}
\sup_{\mathbf{v}_h\in \mathbb{X}_{h}} \frac{ \bigl(\div \mathbf{v}_h,q_h\bigr)}{\|\nab \mathbf{v}_h\|} 
\geq \gamma \|q_h\| 
\qquad \forall q_h\in W_h\, .
\end{equation}
 
\smallskip
Next, let $\rho_h: L^2_{per}(D)\to W_h$ resp.~$\mathcal{R}_h: H^1_{per}(D)/\mathbb{R}\to S_h$ denote the $L^2$-resp.~the Ritz-projection operators which are defined by 
\begin{align}
\label{eq4.8}		
\big(\phi - \rho_h\phi,\chi_h\big) &= 0 \qquad\forall \, \phi\in L^2_{per}(D),\, \chi_h \in {W}_h,\\
 \label{eq4.9}		
\big( \nabla [\psi - \mathcal{R}_h \psi],\nabla\zeta_h\big) &= 0 \qquad\forall \, 
\psi\in H^1_{per}(D)/\mathbb{R},\, \zeta_h \in S_h\, .
\end{align}
Then, the following approximation properties are well known 
(cf. \cite{Ern_Guermond04,Girault_Raviart86,Falk08}):
\begin{align}\label{eq4.2}
\|\phi-\rho_h\phi\| + h\|\nabla(\phi-\rho_h\phi)\| 
&\leq C h^s\|\phi\|_{H^s} \quad \forall \, \phi\in H^s_{per}(D)\, , \\
\|\psi-\mathcal{R}_h\psi\| + h\|\nabla(\psi- \mathcal{R}_h \psi)\| 
&\leq C h^s\|\psi\|_{H^s} \quad \forall \, \psi\in H^s_{per}(D)/\mathbb{R}\,, \label{eq4.2b}
\end{align}
for $s=1,2$. Here, $C$ is a positive constant independent of $h$.

We also consider the space ${\mathbb V}_h\subset {\mathbb X}_h$ of
discretely divergence-free functions, 
\begin{align*}
	{\mathbb V}_h := \bigl\{ {\bf v}_h \in {\mathbb X}_h;\,  (\div {\bf v}_h, q_h ) =0 \quad\forall q_h \in  W_h\bigr\}\, ,
\end{align*}
and define the $L^2_{per}(D; {\mathbb R}^d)$-projection operator  
$\mathbf{P}_h: L^2_{per}(D;\mathbb{R}^d) \rightarrow {\mathbb V}_h$ by
	\begin{align*}
		\bigl({\bf v} - \mathbf{P}_h {\bf v}, {\bf w}_h\bigr) = 0 \qquad\forall \, {\bf v}\in  
	L^2_{per}(D;\mathbb{R}^d),\,   {\bf w}_h\in {\mathbb V}_h\, .
	\end{align*}
The following  approximation properties are well-known (cf. \cite{Rannacher}):
\begin{align} \label{eq4.2c}
\|{\bf v}-\mathbf{P}_h {\bf v}\| + h\|\nabla({\bf v}- \mathbf{P}_h {\bf v})\| 
\leq C h^s\|{\bf v}\|_{H^s} \quad \forall \, {\bf v}\in {\mathbb V} \cap H^s_{per}(D; {\mathbb R}^d)
\end{align}
for $s=1,2$. Here, $C$ is again a positive constant independent of $h$.

\subsection{Formulation of the fully discrete mixed finite element method}\label{sec-4.2} 
The fully discrete, {\em inf-sup} stable finite element below is 
a spatial discretization of Algorithm 1. We note that since ${\mathbb V}_h \not\subset
{\mathbb V}$, in general, the mixed finite element discretization requires improved stability
estimates for the semi-discrete pressure $\{ r^{n+1}\}_n$ as given in Lemma \ref{lem3.1} in 
order to ensure optimal convergence properties.

\medskip
\noindent
{\bf Algorithm 2}

Let ${\bf u}_h^0\in L^2(\Omega; {\mathbb X}_h)$.  For $n=0,1,\ldots, N-1$, we do the following steps:

\smallskip
{\em Step 1:}  Determine ${\xi^{n}_h} \in  L^2(\Omega; S_h)$ by solving 
\begin{equation}\label{eq_WF_xi1b}
	\big(\nab\xi^{n}_h, \nab\phi_h\big) = \big( { {\bf B}({\bf u}^n_h)},\nab\phi_h \big)  
	\qquad \forall\,  \phi_h\in S_h\, .
\end{equation}
\smallskip
{\em Step 2:} Set $\pmb{\eta}^n_h := {\bf B}({\bf u}^n_h)-\nabla \xi^{n}_h$. Find
$({\bf u}^{n+1}_h,r^{n+1}_h) \in L^2\bigl(\Omega, {\mathbb V}_h \times
W_h\bigr)$ by solving 
\begin{subequations}\label{eq_new_h}
	\begin{align}\label{eq_newa_h}
	\bigl({\bf u}^{n+1}_h, {\bf v}_h\bigr) +  & k \bigl(\nabla {\bf u}^{n+1}_h,\nabla {\bf v}_h\bigr)  - k \bigl(\div {\bf v}_h, r^{n+1}_h\bigr) \\ \nonumber
	&= \bigl({\bf u}^n_h,{\bf v}_h\bigr) + k\big(\vf^{n+1}, \mathbf{v}_h\big) + \bigl( \pmb{\eta}^n_h \Delta_{n+1} W , {\bf v}_h\bigr) 
	\qquad\forall \, {\bf v}_h\in {\mathbb X}_h\, ,\\
	\bigl(\div {\bf u}^{n+1}_h,q_h\bigr) &= 0 \qquad\forall \, q_h\in  W_h\,.\label{eq_newb_h}
	\end{align}
\end{subequations}
 
\smallskip
{\em Step 3:} Define the $W_h$-valued random variable {$p^{n+1}_h = r^{n+1}_h + k^{-1} \xi^{n}_h \Delta_{n+1} W  $.}  

\medskip
\begin{remark}
Because of \eqref{eq_WF_xi1b}, we have $(\pmb{\eta}^n_h, \nabla \phi_h) = 0$ for all $\phi_h \in S_h$, ${\mathbb P}$-a.s. We also note that each of {\em Step 1} and {\em Step 2} solves a linear problem which is clearly well-posed; in particular, the well-posedness of \eqref{eq_new_h} is ensured by the {\em inf-sup} property \eqref{inf-sup-TH} of the mixed finite element spaces ${\mathbb X}_h$ and $W_h$. 
\end{remark}

\subsection{Error estimate for the velocity approximation}
The main result of this section is to prove the following optimal estimate for 
the velocity error ${\bf u}^n-{\bf u}^n_h$. 

\begin{theorem}\label{thm4.1}
Suppose that $${\mathbb E}\bigl[\Vert {\bf u}^0 - {\bf u}^0_h\Vert^2\, \bigr] \leq C h^2.$$
Let  $\{ ({\bf u}^n, r^n); 1\leq n\leq N\}$ and 
	$\{ ({\bf u}^n_h, r_h^n); 1\leq n\leq N \}$ be respectively the solutions of Algorithm 1 and 2. Then there exists a constant $C>0$ such that
	\begin{align}\label{eq4.13}
	\max_{1\leq n\leq N} \Bigl({\mathbb E}\bigl[\|{\bf u}^n-{\bf u}^n_h\|^2\, \bigr]\Bigr)^{\frac12}
	+\Bigl({\mathbb E}\Bigl[k\sum_{n=1}^{N}\|\nabla ({\bf u}^n-{\bf u}^n_h)\|^2\, \Bigr] \Bigr)^{\frac12} \leq C\,h\,.
	\end{align}
     
\end{theorem}

\begin{proof}
	Define ${\bf e}^n_{\bf u} = {\bf u}^n - {\bf u}^n_h$ and $e^n_r = r^n - r^n_h$. It 
	is easy to check that $\{({\bf e}^n_{\bf u},e^n_r)\}_n$ satisfies the following error equations ${\mathbb P}$-a.s.~for all tuple
	$({\bf v}_h, q_h) \in {\mathbb X}_h \times W_h$,
	\begin{align}	\label{eq4.17}
	 ({\bf e}^{n+1}_{\bf u}-{\bf e}^n_{\bf u},{\bf v}_h ) + k (\nab {\bf e}^{n+1}_{\bf u},\nab {\bf v}_h )  
	&- k ({e}^{n+1}_r,\div {\bf v}_h ) \\ 
	&= \bigl ( [\pmb{\eta}^{n}-\pmb{\eta}_h^{n}]\Delta_{n+1} W, {\bf v}_h \bigr)\, , \nonumber \\
	\label{eq4.18}	
	\big(\div {\bf e}^{n+1}_{\bf u},q_h\big) &=0\, .
	\end{align}
	Now for any fixed $\omega \in \Omega$, setting ${\bf v}_h = \mathbf{P}_h {\bf e}^{n+1}_{\bf u}(\omega) \in {\mathbb V}_h$ in \eqref{eq4.17} yields   
\begin{align}\label{eq4.17a}
 ({\bf e}^{n+1}_{\bf u} -{\bf e}^{n}_{\bf u}, \mathbf{P}_h {\bf e}^{n+1}_{\bf u}  ) + k (\nab {\bf e}^{n+1}_{\bf u},\nab  \mathbf{P}_h {\bf e}^{n+1}_{\bf u}  ) &-k (e^{n+1}_r,\div \mathbf{P}_h {\bf e}^{n+1}_{\bf u} )   \\
&= \bigl(  [\pmb{\eta}^{n}-\pmb{\eta}_h^{n}]\Delta_{n+1} W, \mathbf{P}_h {\bf e}^{n+1}_{\bf u} \bigr)\, . \nonumber
\end{align}

We now estimate each term on the left-hand side of \eqref{eq4.17a} from below. First, by the 
definition of $\mathbf{P}_h$ we get 
\begin{align}\label{eq4.18a}
\bigl( {\bf e}^{n+1}_{\bf u} -{\bf e}^n_{\bf u}, \mathbf{P}_h {\bf e}^{n+1}_{\bf u} \bigr) &= \bigl( \mathbf{P}_h[{\bf e}^{n+1}_{\bf u} -{\bf e}^n_{\bf u}], \mathbf{P}_h {\bf e}^{n+1}_{\bf u} \bigr) \\
&=\frac12 \Bigl[ \| \mathbf{P}_h {\bf e}^{n+1}_{\bf u} \|^2 - \| \mathbf{P}_h {\bf e}^{n}_{\bf u} \|^2\Bigr]
+\frac12 {\| \mathbf{P}_h [{\bf e}^{n+1}_{\bf u} - {\bf e}^{n}_{\bf u}] \|^2}\, . \nonumber
\end{align}

Next, using again the fact that $\mathbf{P}_h {\bf u}^{n+1}_h = {\bf u}^{n+1}_h$ and Schwarz inequality, we obtain
\begin{align}\label{eq4.18b}
k\bigl(\nab {\bf e}^{n+1}_{\bf u},\nab  \mathbf{P}_h {\bf e}^{n+1}_{\bf u} \bigr) 
&= k \|\nab {\bf e}^{n+1}_{\bf u} \|^2 -k\bigl(\nab {\bf e}^{n+1}_{\bf u}, \nab [{\bf u}^{n+1}- \mathbf{P}_h {\bf u}^{n+1}] \bigr) \\ \nonumber
&\geq \frac{k}2 \|\nab {\bf e}^{n+1}_{\bf u} \|^2 -  \frac{k}2 \|\nab [{\bf u}^{n+1}- \mathbf{P}_h {\bf u}^{n+1}]\|^2 \nonumber \\ \nonumber
&\geq \frac{k}2 \|\nab {\bf e}^{n+1}_{\bf u} \|^2 -C\, k h^2 \|{\bf u}^{n+1}\|_{H^2}^2\, , \nonumber
\end{align}
where we have used \eqref{eq4.2c} to get the last inequality. 

For the next term in (\ref{eq4.17a}), using the fact that $\mathbf{P}_h {\bf e}^{n+1}_{\bf u}$ takes 
values in ${\mathbb V}_h$, and estimates \eqref{eq4.2}, \eqref{eq4.2b}, and \eqref{eq4.2c}, we get
\begin{align}\label{eq4.18c}
-k\bigl(e^{n+1}_r,\div \mathbf{P}_h {\bf e}^{n+1}_{\bf u}\bigr) &= -k\bigl(r^{n+1},\div \mathbf{P}_h {\bf e}^{n+1}_{\bf u}\bigr) \\
& = -k\bigl(r^{n+1} - {\rho_h r^{n+1}},\div \mathbf{P}_h {\bf e}^{n+1}_{\bf u}\bigr)\nonumber  \\
&\geq -k \|r^{n+1} - { \rho_h r^{n+1}}\|  \|\div \mathbf{P}_h {\bf e}^{n+1}_{\bf u}\| 
\nonumber \\
&\geq -C hk \|\nabla r^{n+1}\|  \|\nab {\bf e}^{n+1}_{\bf u}\|  \nonumber \\
&\geq -\frac{k}4 \|\nabla {\bf e}^{n+1}_{\bf u}\|^2 - C^2 h^2 k\|\nabla r^{n+1}\|^2\, .\nonumber
\end{align}

Finally, we bound the only term on the right-hand side of \eqref{eq4.17a} from above.  
By the independence of the increments $\{ \Delta_{n+1} W\}_n$, and its distribution, we get
\begin{align}\label{eq4.18d}
{\mathbb E}\bigl[ \bigl(  [\pmb{\eta}^{n}-\pmb{\eta}_h^{n}]\Delta_{n+1} W, \mathbf{P}_h {\bf e}^{n+1}_{\bf u} \bigr) &\bigr] 
={\mathbb E}\bigl[ \bigl(  [\pmb{\eta}^{n}-\pmb{\eta}_h^{n}]\Delta_{n+1} W, \mathbf{P}_h [{\bf e}^{n+1}_{\bf u} - {\bf e}^{n+1}_{\bf u}]\bigr)\bigr]\\
	 	&  \leq 
	 	k {\mathbb E}\bigl[\| \pmb{\eta}^{n}-\pmb{\eta}_h^{n} \|^2 \bigr] +
		\frac14 {\mathbb E}\bigl[{\| \mathbf{P}_h   ({\bf e}^{n+1}_{\bf u} -{\bf e}^n_{\bf u}) \|^2 } \bigr]\, , \nonumber
\end{align}
and because of (\ref{eq2.6}) and (\ref{eq4.2c}), there holds
\begin{align}\label{eq4.18e}
\| \pmb{\eta}^{n}-\pmb{\eta}_h^{n} \|^2 &\leq 2\| {\bf B}({\bf u}^n)-{\bf B}({\bf u}^n_h)\|^2
+ 2\| \nab (\xi^n -  \xi^n_h)\|^2 \\
&\leq 2C \|{\bf e}^n_{\bf u}\|^2 + 2\| \nab (\xi^n -\xi^n_h)\|^2 \nonumber \\
&=2C \|({\bf u}^n- \mathbf{P}_h {\bf u}^n)+ \mathbf{P}_h {\bf e}^n_{\bf u}\|^2 + 2\| \nab (\xi^n -  \xi^n_h)\|^2 \nonumber \\
&\leq C h^4 \|{\bf u}^n\|_{H^2}^2 + 4 C \|\mathbf{P}_h {\bf e}^n_{\bf u}\|^2 
+ { 2\| \nab (\xi^n -  \xi^n_h)\|^2}\, .  \nonumber 
\end{align}	 	
	 	
To control $\| \nab (\xi^n -\xi^n_h)\|$, we recall the definitions of $\xi^n$ 
and $\xi^n_h$ to get 
\begin{align*} 
\bigl( \nab [\xi^n -\xi^n_h], \nabla \phi_h \bigr) 
&= \bigl({\bf B}({\bf u}^n)-{\bf B}({\bf u}^n_h), \nabla \phi_h \bigr) \qquad \forall \phi_h\in S_h\, .
\end{align*}	 
Setting $\phi_h = \mathcal{R}_h [\xi^n - \xi^n_h]= (\xi^n -\xi^n_h) 
- (\xi^n-\mathcal{R}_h \xi^n )$, properties (\ref{eq2.6a}) and (\ref{eq4.2b}) yield
\begin{align*} 
\| \nab (\xi^n -\xi^n_h)\|^2 
&\leq  \bigl( \nab [\xi^n -  \xi^n_h], \nabla [\xi^n-\mathcal{R}_h \xi^n]\bigr)
 + C\|{\bf e}^n_{\bf u}\| \|\nab (\xi^n -\xi^n_h)\| \nonumber \\
&\leq \frac12 \| \nab (\xi^n -\xi^n_h) \|^2 
+ C h^2\|\xi^n\|_{H^2/\mathbb{R}}^2  + C \|{\bf e}^n_{\bf u}\|^2\, .  \nonumber 
\end{align*}	 
Hence, by (\ref{rem-1a}) in Remark \ref{remy1}, 
(\ref{eq4.2c}), and \eqref{eq2.6c}, we get
\begin{align} \label{eq4.18f}
\| \nab (\xi^n -\xi^n_h)\|^2 &\leq C h^2\|\xi^n\|_{H^2/\mathbb{R}}^2  
+ C\|{\bf e}^n_{\bf u}\|^2  \\ 
&\leq C h^2 \|\div {\bf B}({\bf u}^n) \|^2 
+ C \|{\bf e}^n_{\bf u}\|^2  \nonumber \\\nonumber
&\leq C h^2 \|\nab \bu^n \|^2 
+ Ch^4 \|\bu^n\|_{H^2}^2 + C\|\bP_h \be^n_{\mathbf{u}}\|^2.
\end{align}	 
Therefore,
\begin{align}\label{eq4.19}
	\|\pmb{\eta}^n - \pmb{\eta}_h^n\|^2 \leq C \Bigl( h^2   \| \nabla {\bf u}^n\|^2  
	+ h^4 \|{\bf u}^n\|_{H^2}^2
	+  \|\mathbf{P}_h {\bf e}^n_{\bf u}\|^2 \Bigr)\, .
\end{align}

We insert estimates \eqref{eq4.18a}--\eqref{eq4.19} into \eqref{eq4.17a}, take the expectation, 
and apply the summation operator $\sum_{n=0}^{m}$ for any $0\leq m \leq N-1$ to conclude 
\begin{align}\label{eq4.25}
	\frac{1}{2}{\mathbb E}\bigl[\| \mathbf{P}_h {\bf e}^{m+1}_{\bf u}\|^2\bigr] 
	&+ \frac{1}{4}\sum_{n=0}^{m}{\mathbb E}\bigl[\| \mathbf{P}_h ({\bf e}^{n+1}_{\bf u} - {\bf e}^{n}_{\bf u}) \|^2 \bigr] + \frac{1}{4} {\mathbb E}\big[k\sum_{n=0}^{m}\|\nab {\bf e}^{n+1}_{\bf u}\|^2\big] \\ \nonumber
    &\leq C k\sum_{n=0}^{m} {\mathbb E}\bigl[ \| \mathbf{P}_h {\bf e}^{n}_{\bf u} \|^2 \bigr] 
	+ C\,h^2. \nonumber
\end{align}
	 
Applying the discrete Gronwall inequality to \eqref{eq4.25} then leads to
\begin{align}\label{eq4.25b}
	\frac{1}{2} {\mathbb E}\bigl[\|\mathbf{P}_h {\bf e}^{m+1}_{\bf u} \|^2\bigr] &+ \frac{1}{4}\sum_{n=0}^{m} {\mathbb E}\bigl[\|\mathbf{P}_h ({\bf e}^{n+1}_{\bf u} - {\bf e}^{n}_{\bf u}) \|^2 \bigr] 
	+ \frac{1}{4} {\mathbb E}\big[k\sum_{n=0}^{m}\|\nab {\bf e}^{n+1}_{\bf u}\|^2\big] \\\nonumber 
	&\leq  \exp(C\,T)\, Ch^2 \qquad (1\leq m \leq N)\,.
\end{align}
 
Finally, the desired estimate \eqref{eq4.13}  
follows from an application of the triangle inequality on 
${\bf e}^{m+1}_{\bf u} = ({\bf u}^{m+1} -\mathbf{P}_h {\bf u}^{m+1}) + \mathbf{P}_h {\bf e}^{m+1}_{\bf u}$ 
and using \eqref{eq4.25b} and \eqref{eq4.2c}. The proof is complete. 
\end{proof}

\subsection{Error estimates for the pressure approximations}\label{sec-4.4}
In this subsection, we derive some error estimates for both, $r^n-r^n_h$ and $p^n-p^n_h$. The argumentation parallels the one in section \ref{sec-3.4}, and uses
the {\em inf-sup} condition  (\ref{inf-sup-TH}),  in particular.
 
\begin{theorem}\label{thm4.2}
	Suppose that $${\mathbb E}\bigl[\Vert {\bf u}^0 - {\bf u}^0_h\Vert^2\, \bigr] \leq C h^2.$$
Let  $\{ ({\bf u}^n, r^n); 1\leq n\leq N\}$ and 
	$\{ ({\bf u}^n_h, r_h^n); 1\leq n\leq N \}$ be respectively the solutions of Algorithm 1 and 2.
	There exists a constant $C>0$,
	such that 
	\begin{align*}\label{eq4.22}
	\bigl(\mathbb{E}\bigl[ \bigl\|k\sum^N_{n=1} (r^n- r^n_h) \bigr\|^2\, \bigr]\bigr)^{\frac12}
	\leq C\,  h\, .
	\end{align*}
\end{theorem}

\begin{proof}
	Summing \eqref{eq4.17} (after lowering the index by one) over $1\leq n\leq m  \leq N $ leads to	
	\begin{align*}
	\bigl(\beu^m,\bv_h\bigr) &+ k\sum_{n=1}^m \bigl(\nab \beu^n, \nab \bv_h\bigr) 
	-k\sum_{n=1}^m \bigl(\div \bv_h, e^n_r \bigr)\\
	&\qquad=(\beu^0, \bv_h)+ \sum_{n=1}^m \bigl([\peta^{n-1}-\peta_h^{n-1}]\Delta_{n} W,\bv_h\bigr)
	\qquad\forall \bv_h\in \mathbb{X}_h\, .
	\end{align*}
By (\ref{inf-sup-TH}), we conclude
(compare with (\ref{pressure-e})) 
\begin{equation*}	\frac{1}{3\gamma} \Vert k \sum_{n=1}^m e^n_r\Vert \leq 
	\Vert {\bf e}^m_{\bf u}\Vert + \Vert {\bf e}^0_{\bf u}\Vert + \bigl\Vert  k \sum_{n=1}^m \nabla
	{\bf e}^m_{\bf u}\bigr\Vert + \bigl\Vert \sum_{n=1}^m (\pmb{\eta}^{n-1} -
	\pmb{\eta}^{n-1}_h) \Delta_{n+1}W\bigr\Vert\, .
	\end{equation*}
	Taking expectations after squaring both sides, using estimate (\ref{eq4.18e}) and
	Theorem \ref{thm4.1} yield the desired result.
\end{proof}

The following result now is a simply corollary of Theorem \ref{thm4.2}.
\begin{corollary}\label{cor4.4}
	Let $\{p^n_h\}_n$ be the solution in Algorithm 2. Then there exists a 
	constant $C>0$, 
	such that 
	\begin{align*}
	\Bigl(\mathbb{E}\bigl[ \bigl\|k\sum^N_{n=1} (p^n- p^n_h) \bigr\|^2\, \bigr]\Bigr)^{\frac12}
	 \leq C\, h\, .
	\end{align*}
\end{corollary}
 
\subsection{Space-time error estimates for Algorithm 2}\label{sec-4.5}
Theorems \ref{thm3.3}, \ref{thm3.4a}, \ref{thm4.1}, \ref{thm4.2}, 
and Corollaries \ref{thm3.4}, \ref{cor4.4} 
now provide the following global error estimates.

\begin{theorem}\label{thm4.5}
Let $({\bf u}, P)$ solve (\ref{eq1.1}), and $\{ ({\bf u}^n_h, r_h^n, p_h^n); 1\leq n\leq N \}$ solves Algorithm 2. There exists a constant $C>0$, 
	\begin{align*}
	{\rm (i)}\ & \max_{1\leq n\leq N} \Bigl(\mathbb{E}\bigl[\|\bu(t_n)-\bu^n_h\|^2\,\bigl]\Bigr)^{\frac12}
	 +\Bigl( \mathbb{E}\Bigl[ k\sum_{n=1}^{N}\|\nabla (\bu(t_n)-\bu^n_h)\|^2\,\Bigr] \Bigr)^{\frac12} 
	\leq C \big( k^{\frac12}+  \,  h\bigr)\,, \\ 
	{\rm (ii)}\ & \Bigl(\mathbb{E} \bigl[ \bigl\|R(t_m) -k\sum^m_{n=1}r^n_h \bigr\|^2\, \bigr]\Bigr)^{\frac12} 
	+ 
	\Bigl( \mathbb{E}\bigl[\bigl\|P(t_m) -k\sum^m_{n=1}p^n_h \bigr\|^2\,\bigr] \Bigr)^{\frac12}
\leq C \bigl(k^{\frac12}+ h\bigr)\, ,
	\end{align*}
	for all $1 \leq m \leq N$.
\end{theorem}

  
\section{Stabilization methods for (\ref{eq1.1})}\label{sec5}\label{sec-5}
The scheme in section \ref{sec-4} requires {\em inf-sup} stable pairings $({\mathbb X}_h, W_h)$, 
for which the Taylor-Hood mixed finite element is one example. 
By recalling its definition in subsection \ref{sec-4.1}, we observe that the dimension
of ${\mathbb X}_h$ exceeds that of $W_h$. The motivation for the stabilization methods in  
\cite{hughes} is to relax the  {\em inf-sup} stability criterion for pairings of ansatz spaces 
in order to allow for equal-order ansatz spaces for both, velocity and pressure approximates; see 
\cite{hughes, Brezzi_Fortin91, Girault_Raviart86, Ern_Guermond04} for further details. 
 
Below we replace ${\mathbb X}_h$ defined in subsection \ref{sec-4.1} by 
$${\mathbb Y}_{h} =\bigl\{{\bf v}_h\in H^1_{per}( {D}; {\mathbb R}^d);\, {\bf v}_h|_K\in P_1(K, {\mathbb R}^d)\,\,\forall \, K\in  \cT_h\bigr\}\,,$$
to which we associate the $L^2_{per}$-projection operator ${\bf Q}_h:
L^2_{per}(D; {\mathbb R}^d) \rightarrow {\mathbb Y}_{h}$ by
$$({\bf v} - {\bf Q}_h {\bf v}, {\bf w}_h) = 0 
\qquad \forall\, {\bf v} \in L^2_{per}(D; {\mathbb R}^d), \ {\bf w}_h \in {\mathbb Y}_{h}\,,$$
which satisfies the following approximation property (cf. \cite{Ern_Guermond04}):
\begin{align} \label{eq4.100c}
\|{\bf v}-\mathbf{Q}_h {\bf v}\| + h\|\nabla({\bf v}- \mathbf{Q}_h {\bf v})\| 
\leq C h^s\|{\bf v}\|_{H^s} \quad \forall \, {\bf v}\in  H^s_{per}(D; {\mathbb R}^d)
\end{align}
for $s=1,2$. Here, $C>0$ is a constant independent of $h$. 
Moreover, let $W_h$ be the same as in section \ref{sec-4},  and $\widetilde{\mathcal R}_h$ 
denote the Ritz projection from $H^1_{per}(D)/\mathbb{R}$ to $W_{h}$.
Again, we take $S_h=W_h$ in this section.

In this section, we consider the equal-order pairing $({\mathbb Y}_{h}, W_{h})$ to discretize
(\ref{eq1.1}) based on \eqref{eq_new_intro_ref_2a}--\eqref{eq_new_intro_ref_2b}, which violates the {\em inf-sup} condition; in fact, the following estimate is known to hold (cf. \cite{hughes}): there exists $\delta >0$ independent of $h>0$, such that 
\begin{equation}\label{lbb_disk}
\frac{1}{\delta^2} \Vert q_h\Vert^2 \leq \sup_{{\bf v}_h \in {\mathbb Y}_{h}}
\frac{\vert(q_h, {\rm div}\, {\bf v}_h)\vert^2}{\Vert \nabla {\bf v}_h\Vert^2}
+ h^2 \Vert \nabla q_h\Vert^2 \qquad \forall\, q_h \in W_{h}\, .
\end{equation}
\eqref{lbb_disk} can be regarded as the reason why this pairing still performs optimally
when applied to the Stokes problem,  
where $\varepsilon = {\mathcal O}(h^2)$.  Below we 
show that such a strategy can be again successful for the stochastic 
Stokes problem (\ref{eq1.1}), 
if the {\em proper} pressure is chosen for the perturbation, and that 
using the Helmholtz projection of the noise provides such an approach.  

To prepare for the analysis, we start with a modification of Algorithm 1 that perturbs 
the incompressibility constraint.

\medskip
\noindent
{\bf Algorithm 3}

Let $0<\varepsilon \ll 1$ and ${\bf u}_{\varepsilon}^0 = {\bf u}_0$.  For $n=0,1,\ldots, N-1$, do
the following steps:

\smallskip
{\em Step 1:}  Find $\xi^{n}_{\varepsilon} \in  L^2\bigl(\Omega,H^1_{per}(D)/\mathbb{R}\bigr)$ by solving 
\begin{equation}\label{eq_WF_xi1bc}
	\big(\nab\xi^{n}_{\varepsilon}, \nab\phi\big) = \big(  {\bf B}({\bf u}^n_{\varepsilon}),\nab\phi \big)  
	\qquad \forall \phi\in H^1_{per}(D)\, .
\end{equation}

\smallskip
{\em Step 2:} Set $\pmb{\eta}^n_{\varepsilon} := {\bf B}({\bf u}^n_\varepsilon)-\nabla \xi^{n}_\varepsilon$, and find
$({\bf u}^{n+1}_\varepsilon,r^{n+1}_\varepsilon) \in L^2\bigl(\Omega, H^1_{per}(D; {\mathbb R}^d) \times
H^1_{per}(D)/\mathbb{R}\bigr)$ by solving 
\begin{subequations}\label{eq_new_hc}
	\begin{align}
	&\big({\bf u}^{n+1}_\varepsilon,{\bf v}\big) + k \big(\nabla {\bf u}^{n+1}_\varepsilon,\nabla {\bf v}\big)  - k \big(\div {\bf v}, r^{n+1}_{\varepsilon}\big)  \label{eq_newa_hc} \\
	&\hskip 0.5in 
	= \big({\bf u}^{n}_\varepsilon, {\bf v}\big) + k\big(\vf^{n+1},\mathbf{v}\big) + \big(\pmb{\eta}^n_{\varepsilon}\Delta_{n+1}W, {\bf v} \bigr) 
	\qquad\forall \, {\bf v}\in H^1_{per}(D; {\mathbb R}^d)\, ,  \nonumber\\
	&\bigl(\div {\bf u}^{n+1}_\varepsilon,q\bigr) + {\varepsilon \bigl(\nabla r^{n+1}_{\varepsilon},\nabla q\bigr)} = 0 \qquad\forall q\in H^1_{per}(D)\, .\label{eq_newb_hc}
	\end{align}
\end{subequations}

\smallskip
{\em Step 3:} Define $p^{n+1}_\varepsilon := r^{n+1}_\varepsilon +k^{-1} \xi^{n}_\varepsilon\Delta_{n+1} W$.  

\medskip
Because each step involves a coercive linear problem, Algorithm 3 has a
unique solution. The first energy estimate can be obtained from (\ref{eq_new_hc}) by fixing 
one $\omega \in \Omega$ and choosing
$( {\bf v}, q ) = \bigl({\bf u}^{n+1}_\varepsilon(\omega), r^{n+1}_\varepsilon(\omega)\bigr)$, we then 
obtain the identity
\begin{align}\label{hhelp0}
\frac{1}{2}\Bigl(& \Vert {\bf u}^{n+1}_{\varepsilon}\Vert^2 - \Vert {\bf u}^{n}_{\varepsilon}\Vert^2 + \Vert {\bf u}^{n+1}_{\varepsilon} - {\bf u}^{n}_{\varepsilon}\Vert^2\Bigr) \\ \nonumber
&+ \frac{k}{2} \Vert \nabla {\bf u}^{n+1}_{\varepsilon}\Vert^2 + k \varepsilon
\Vert \nabla r^{n+1}_{\varepsilon}\Vert^2 = \frac{Ck}{2}\|\vf^{n+1}\|^2 + \bigl( \pmb{\eta}^{n}_\varepsilon \Delta_{n+1} W, {\bf u}^{n+1}_{\varepsilon}\bigr)\, .
\end{align}
Taking expectations, applying the summation operator $\sum_{n=0}^{m}$ for any $0\leq m \leq N-1$, and using
 the independence of the the increments $\{\Delta_{n+1}W\}_n$ yield
\begin{eqnarray}\label{hhelp1} {\mathbb E}\bigl[ \bigl( \pmb{\eta}^{n}_\varepsilon \Delta_{n+1} W, {\bf u}^{n+1}_{\varepsilon}\bigr)\bigr] &=& {\mathbb E}\bigl[ \bigl( \pmb{\eta}^{n}_\varepsilon \Delta_{n+1} W, [{\bf u}^{n+1}_{\varepsilon}- {\bf u}^{n}_{\varepsilon}]\bigr)\bigr] \\ \nonumber
&\leq& k {\mathbb E}\bigl[ \Vert \pmb{\eta}^{n}_{\varepsilon}\Vert^2\bigr] + \frac{1}{4} {\mathbb E}\bigl[ \Vert {\bf u}^{n+1}_{\varepsilon}- {\bf u}^{n}_{\varepsilon} \Vert^2\bigr]\, .
\end{eqnarray}
Because of (\ref{eq_WF_xi1bc}) and (\ref{eq2.6b}), we have  $${\mathbb E}\bigl[ \Vert \nabla \xi^n_\varepsilon\Vert^2\bigr] \leq {\mathbb E}\bigl[ \Vert
{\bf B}({\bf u}^{n}_{\varepsilon})\Vert^2\bigr] \leq C(1+ {\mathbb E}[\Vert {\bf u}^{n}_{\varepsilon}\Vert^2)\,,$$
and therefore ${\mathbb E}[\Vert \pmb{\eta}^n_\varepsilon\Vert^2] \leq 
C(1+ {\mathbb E}[\Vert {\bf u}^{n}_{\varepsilon}\Vert^2)$ in (\ref{hhelp1}).
We insert these auxiliary estimates into (\ref{hhelp0}), take expectations, 
sum over all iteration steps, and use the discrete Gronwall inequality  to get 
\begin{align}\label{hhelp2}
\max_{0 \leq n \leq N-1} \frac{1}{2} {\mathbb E}\bigl[ \Vert {\bf u}^{n+1}_\varepsilon\Vert^2\bigr] 
&+ k\sum_{n=0}^{N-1} {\mathbb E}\bigl[ \Vert \nabla {\bf u}^{n+1}_{\varepsilon}\Vert^2 + \varepsilon \Vert \nabla r^{n+1}_\varepsilon\Vert^2\bigr] \\
& \leq C \Bigl( {\mathbb E}\bigl[\Vert {\bf u}_0 \Vert^2\bigr] 
 + k\sum_{n=0}^{N-1}  {\mathbb E}\bigl[ \|\vf^{n+1}\|^2 \bigr]  \Bigr)
\, . \nonumber
\end{align}

Note that the estimate for $\{ \nabla r^{n+1}\}$ is scaled by $\varepsilon >0$, which is too weak in the following to verify optimal error estimates for a spatial discretization of Algorithm 3. 
The following stability result therefore sharpens the estimate (\ref{hhelp2}); its proof 
crucially exploits again the fact that each $\pmb{\eta}^n_\varepsilon$ is a ${\mathbb H}$-valued 
random variable.

\begin{lemma}\label{stabilty-thm}
	 Let $\{ (u^n_\varepsilon,r^n_\varepsilon) \}_n$ be 
	 the solution of Algorithm 3. Then there exists a constant $C>0$, 
	 such that 
\begin{align}\label{stability-est} 
\max_{1\leq n\leq N} {\mathbb E}\bigl[ \|\nabla {\bf u}^n_\varepsilon\| ^2 \bigr]
+ k\sum_{n=1}^N {\mathbb E}\bigl[ \|\Delta {\bf u}^n_\varepsilon\| ^2]  + k\sum_{n=1}^N  {\mathbb E}\bigl[ {\|\nab r^n_{\varepsilon}\| ^2 } \bigr] &\leq C\, 
\end{align}

\end{lemma}

\begin{proof} 
	{\em Step 1:}  We adapt the argumentation from \cite[Thm.~3.1]{CHP12}, and interpret problem (\ref{eq_new_intro_ref_2}) --- with $\pmb{\eta}^n$ being replaced by
$\pmb{\eta}^n_\varepsilon$ --- as perturbation of problem
(\ref{eq_new_intro_ref}) 
Subtracting the corresponding equations of both systems and denoting
${\bf e}^{n+1}_{\bf u} := {\bf u}^{n+1} - {\bf u}^{n+1}_\varepsilon$ resp.~$e^{n+1}_r := r^{n+1} - r^{n+1}_{\varepsilon}$ yield
\begin{subequations}
\begin{alignat}{2}\label{eq_new_intro_ref_3a}
	 {\bf e}_{\bf u}^{n+1} - k \Delta {\bf e}_{\bf u}^{n+1} + k \nabla e_{r}^{n+1}  & =  {\bf e}_{\mathbf{u}}^n   +  [\pmb{\eta}^n- \pmb{\eta}^n_\varepsilon] \Delta_{n+1} W 
	 &&\qquad\mbox{in } D \,,  \\ \label{eq_new_intro_ref_3b}
	 \div {\bf e}_{\bf u}^{n+1} - \varepsilon \Delta e^{n+1}_r &= - \varepsilon \Delta r^{n+1}  
	 &&\qquad\mbox{in } D\, .
\end{alignat}
\end{subequations}
Now fix one $\omega \in \Omega$, test (\ref{eq_new_intro_ref_3a}) with
${\bf v} = {\bf e}_{\bf u}^{n+1}(\omega)$, and (\ref{eq_new_intro_ref_3b}) with $q = e^{n+1}_r(\omega)$, and afterwards sum both equations,  we then conclude
\begin{eqnarray}\label{error_stab}
&& \frac{1}{2} \Bigl( \Vert {\bf e}_{\bf u}^{n+1}\Vert^2 - \Vert {\bf e}_{\bf u}^{n}\Vert^2 +
\Vert {\bf e}_{\bf u}^{n+1} - {\bf e}_{\bf u}^{n}\Vert^2 \Bigr) +  k \Vert \nabla {\bf e}_{\bf u}^{n+1}\Vert^2 
+k \varepsilon\Vert \nabla {e}_{r}^{n+1}\Vert^2 \\ \nonumber
&&\qquad = \bigl( [\pmb{\eta}^n - \pmb{\eta}^n_{\varepsilon}] \Delta_{n+1} W, {\bf e}_{\bf u}^{n+1}\bigr) + k \varepsilon (\nabla r^{n+1}, \nabla e^{n+1}_r)\, .
\end{eqnarray}
Using Young's inequality, hiding one part of the last term into the corresponding term 
on the left-hand side, using the independence of $ \Delta_{n+1}W$ and ${\bf e}^n_{\bf u}$, 
$ \Delta_{n+1}W$ as well as of $[\pmb{\eta}^n - \pmb{\eta}^n_{\varepsilon}]$, and utilizing 
 (\ref{eq2.6}), we obtain
\begin{eqnarray*}
&&{\mathbb E}\bigl[ \bigl( [\pmb{\eta}^n - \pmb{\eta}^n_{\varepsilon}] \Delta_{n+1} W, {\bf e}_{\bf u}^{n+1}\bigr)\bigr] =
{\mathbb E}\bigl[ \bigl( [\pmb{\eta}^n - \pmb{\eta}^n_{\varepsilon}] \Delta_{n+1} W, {\bf e}_{\bf u}^{n+1}- {\bf e}_{\bf u}^{n}\bigr)\bigr] \\
&&\qquad \leq k {\mathbb E}\bigl[ \Vert \pmb{\eta}^n - \pmb{\eta}^n_{\varepsilon}\Vert^2\bigr] + \frac{1}{2} {\mathbb E}\bigl[ \Vert {\bf e}_{\bf u}^{n+1}- {\bf e}_{\bf u}^{n} \Vert^2\bigr] \\
&&\qquad \leq C k\Bigl({\mathbb E}[\Vert {\bf e}^{n}_{\bf u}\Vert^2]+ 
{\mathbb E}\bigl[\Vert \nabla (\xi^n - \xi^n_{\varepsilon})\Vert^2\bigr]\Bigr) + \frac{1}{2} {\mathbb E}\bigl[ \Vert {\bf e}_{\bf u}^{n+1}- {\bf e}_{\bf u}^{n} \Vert^2\bigr]\, .
\end{eqnarray*}
Subtracting (\ref{eq_WF_xi1bc}) from (\ref{eq_WF_xi1a}) and using (\ref{eq2.6}), we get 
        \begin{equation*}
        \Vert \nabla (\xi^n - \xi^n_{\varepsilon})\Vert \leq \Vert {\bf B}({\bf u}^n) - {\bf B}({\bf u}^{\varepsilon})\Vert \leq C \Vert {\bf e}^{n}_{\bf u}\Vert\, .
        \end{equation*}
        
 Hence, we then conclude from (\ref{error_stab}) with the help of the discrete Gronwall inequality, 
 and (\ref{eq3.4a}) that
 \begin{eqnarray*}
&& \frac{1}{2}  {\mathbb E}\bigl[ \Vert {\bf e}_{\bf u}^{m+1}\Vert^2\bigr] +
\sum_{n=0}^m {\mathbb E}\bigl[\Vert {\bf e}_{\bf u}^{n+1} - {\bf e}_{\bf u}^{n}\Vert^2
\bigr]   +  k \sum_{n=0}^m {\mathbb E}\bigl[ \Vert \nabla {\bf e}_{\bf u}^{n+1}\Vert^2 
+ \varepsilon\Vert \nabla {e}_{r}^{n+1}\Vert^2\bigr] \\
&&\qquad \leq C \varepsilon k\sum_{n=0}^m{\mathbb E}\bigl[ \Vert \nabla r^{n+1}\Vert^2 \bigr] \leq C \varepsilon
 \qquad (0 \leq m \leq N)\, .
\end{eqnarray*}
Consequently, by (\ref{eq3.4a}) we conclude that 
$$
k \sum_{n=0}^N {\mathbb E}\bigl[\Vert \nabla e^{n+1}_r\Vert^2\,\bigr]\leq C
\qquad \mbox{implies} \qquad
k \sum_{n=0}^N {\mathbb E}\bigl[\Vert \nabla r^{n+1}_{\varepsilon}\Vert^2\,\bigr]\leq C.
$$ 
  
\medskip
{\em Step 2:} 
Fix one $\omega \in \Omega$ in (\ref{eq_new_intro_ref_2a}), multiply the equation with 
$-\Delta {\bf u}^{n+1}_\varepsilon(\omega)$, integrate,  perform an integration by 
	parts on the last term, and use the periodicity of $\pmb{\eta}^n_{\varepsilon}$ and ${\bf u}^{n+1}_\varepsilon$, we get  
\begin{eqnarray}\label{last_term_1}
&&\frac{1}{2}{\mathbb E}\Bigl[ \Vert \nabla {\bf u}^{n+1}_\varepsilon\Vert^2 -
\Vert \nabla {\bf u}^{n}_\varepsilon\Vert^2 +
\Vert \nabla ({\bf u}^{n+1}_\varepsilon -  {\bf u}^{n}_\varepsilon)
\Vert^2 + k\Vert \Delta {\bf u}^{n+1}_\varepsilon\Vert^2
\Bigr]\\
&&\qquad\qquad \leq k{\mathbb E}\bigl[ \Vert \nabla r^{n+1}_\varepsilon\Vert^2\bigr]
+ \mE\Bigl[ \int_D \Delta_{n+1}W \nab \pmb{\eta}^n_{\varepsilon} \cdot \nab ({\bf u}^{n+1}_\varepsilon - {\bf u}^{n}_\varepsilon ) \, d{\bf x} \Bigr] \, . \nonumber
\end{eqnarray}

To bound the last term above, we use Schwarz inequality, the fact that 
$\pmb{\eta}^{n}_{\varepsilon} = \mathbf{B}(\bu^{n}_{\varepsilon}) - \nab\xi^n_{\varepsilon}$, \eqref{eq2.6}, and \eqref{rem-1a} to get  
\begin{align} \label{last_term_2}
\mE\Bigl[ \int_D \Delta_{n+1}W \nab \pmb{\eta}^n_{\varepsilon} \cdot \nab ({\bf u}^{n+1}_\varepsilon - {\bf u}^{n}_\varepsilon ) \, d{\bf x} \Bigr] 
&\leq \frac14 \mE\bigl[ \Vert \nabla ( {\bf u}^{n+1}_\varepsilon -
 {\bf u}^{n}_\varepsilon) \Vert^2 \bigr]  \\
&\qquad + C k \mE\bigl[ \| \nabla \bu^n \|^2  \bigr].  \nonumber
\end{align}
Substituting \eqref{last_term_2} into \eqref{last_term_1}, summing over all time steps,  
and using \eqref{hhelp2} and the result of {\em Step 1}, we get the desired estimate \eqref{stability-est}. The proof is complete. 
\end{proof}

From {\em Step 1} of the above proof, we also obtain the following result.

\begin{theorem}\label{error-thm}
Let $\{ ({\bf u}^{n+1}, r^{n+1})\}_n$ and $\{ ({\bf u}^n_\varepsilon,r^n_\varepsilon) \}_n$  
be the solutions  of {\rm Algorithm 1} and {\rm 3},
respectively. Then there exists a constant $C>0$,
such that 
	\begin{align}\label{error-est} 
	\max_{1\leq n\leq N} {\mathbb E}\bigl[ \|{\bf u}^n - {\bf u}^n_\varepsilon\|^2 \bigr] 
	 &+ k\sum_{n=1}^N {\mathbb E}\bigl[ \|\nabla( {\bf u}^n- {\bf u}^n_\varepsilon)\|^2 \\
	 &+ \varepsilon \|\nab(r^n- r^n_{\varepsilon})\|^2\bigr]  \leq C \varepsilon\, .\nonumber
	\end{align}
\end{theorem}

We are ready to bound the error between the pressures $\{r^n\}_n$ and $\{r^n_\varepsilon\}_n$; 
the proof of it uses (\ref{eq_new_intro_ref_3a}) after summation in time, and follows along the 
lines of the proof of Theorem \ref{thm3.4a}, using the stability of the divergence operator (cf.~estimate (\ref{pressure-e})), and Theorem \ref{error-thm}. 

\begin{theorem}\label{thm3.4aaa}
	Let $\{r^n; 1\leq n\leq N\}$ be generated by {\rm Algorithm 1}  
	and $\{r^n_\varepsilon; 1\leq n\leq N\}$ by {\rm Algorithm 3}.   
		There exists a constant $C>0$, 
		such that for $m=1,2,\cdots, N$,
	\begin{align*}
	\biggl(\mathbb{E}\Bigl[ \|k\sum^m_{n=1}r^n-k\sum^m_{n=1}r^n_\varepsilon  \|^2 + \|k\sum^m_{n=1}p^n-k\sum^m_{n=1}p^n_\varepsilon  \|^2\Bigr] \biggr)^{\frac12}
	\leq C  \varepsilon\, .
	\end{align*}
\end{theorem}

Next, we present the following modification of Algorithm 2. 

\medskip
\noindent
{\bf Algorithm 4}

Let $0<\varepsilon \ll 1$ and ${\bf u}_{\varepsilon,h}^0\in L^2(\Omega; {\mathbb Y}_{h})$.  
For $n=0,1,\ldots, N-1$, do the following steps:

\medskip
{\em Step 1:}    Determine ${\xi^{n}_{\varepsilon,h}} \in  L^2(\Omega; S_{h})$ from 
\begin{equation}\label{eq_WF_xi1bd}
	\big(\nab\xi^{n}_{\varepsilon,h}, \nab\phi_h\big) = \big( { {\bf B}({\bf u}^n_{\varepsilon,h})},\nab\phi_h \big)  
	\qquad \forall\,  \phi_h\in S_{h}\, .
\end{equation}
\smallskip
{\em Step 2:} Set $\pmb{\eta}^n_{\varepsilon,h} := {\bf B}({\bf u}^n_{\varepsilon,h})-\nabla \xi^{n}_{\varepsilon,h}$. Find
$({\bf u}^{n+1}_{\varepsilon,h},r^{n+1}_{\varepsilon,h}) \in L^2\bigl(\Omega, {\mathbb Y}_{h} \times W_{h} \bigr)$ by solving 
\begin{subequations}\label{eq_new_hd}
	\begin{align}\label{eq_newa_hd}
	 ({\bf u}^{n+1}_{\varepsilon,h}, {\bf v}_h ) +  & k  (\nabla {\bf u}^{n+1}_{\varepsilon,h},\nabla {\bf v}_h )  - k  (\div {\bf v}_h, r^{n+1}_{\varepsilon,h})   \\ \nonumber
	&=  ({\bf u}^n_{\varepsilon,h},{\bf v}_h ) + k\big(\vf^{n+1} , \mathbf{v}_h\big)+ ( \pmb{\eta}^n_{\varepsilon,h} \Delta_{n+1} W , {\bf v}_h ) 
	\qquad\forall \, {\bf v}_h\in {\mathbb Y}_{h}\, ,\\
	 (\div {\bf u}^{n+1}_{\varepsilon,h},&q_h ) + \varepsilon (\nabla {\bf u}^{n+1}_{\varepsilon,h}, \nabla q_h) = 0 \qquad\forall \, q_h\in  W_{h}\,.\label{eq_newb_hd}
	\end{align}
\end{subequations}
 
\smallskip
{\em Step 3:} Define the $W_{h}$-valued random variable {$p^{n+1}_{\varepsilon,h} = r^{n+1}_{\varepsilon,h} + k^{-1}\xi^{n}_{\varepsilon,h} \Delta_{n+1} W $.}  

\medskip
The main result of this section is the following estimate for the velocity error 
${\bf u}^n_{\varepsilon} -{\bf u}^n_{\varepsilon,h}$. 

\begin{theorem}\label{error-thm-h}
Suppose $${\mathbb E}\bigl[\Vert {\bf u}^0 - {\bf u}^0_{\varepsilon,h}\Vert^2 \bigr] 
\leq Ch^2.$$
Let  $\{ ({\bf u}^n_\varepsilon, r^n_\varepsilon); 1\leq n\leq N\}$ and 
	$\{ ({\bf u}^n_{\varepsilon,h}, r_{\varepsilon,h}^n); 1\leq n\leq N \}$ be the solutions 
	of {\rm Algorithm  3} and {\rm 4}, respectively. 	Then there exists 
	a constant $C>0$, such that 
	\begin{align}\label{eq4.13b}
	\max_{1\leq n\leq N} \Bigl({\mathbb E}\bigl[\|{\bf u}^n_\varepsilon-{\bf u}^n_{\varepsilon,h}\|^2 \bigr]\Bigr)^{\frac12}
	&+\Bigl({\mathbb E}\Bigl[k\sum_{n=1}^{N}\|\nabla ({\bf u}^n_\varepsilon-{\bf u}^n_{\varepsilon,h})\|^2 \Bigr] \Bigr)^{\frac12} \\
	&\leq C\,\Bigl(h  + \frac{h^2}{\sqrt{\varepsilon}}\Bigr)\,. \nonumber
	\end{align}
	This estimate suggests that 
	$\varepsilon = {\mathcal O}(h^2)$ is the optimal choice of $\varepsilon$.
  
\end{theorem}

\begin{proof}
Let ${\bf e}^{n+1}_{\bf u} := {\bf u}^{n+1}_\varepsilon - {\bf u}^{n+1}_{\varepsilon,h}$ 
and $e^{n+1}_r := r^{n+1}_\varepsilon - r^{n+1}_{\varepsilon,h}$. 
Then $ \{({\bf e}^n_{\bf u},e^n_r)\}_n$ satisfies the following error equations 
${\mathbb P}$-a.s.~for all tuple $({\bf v}_h, q_h) \in {\mathbb Y}_{h} \times W_{h}$,
	\begin{align}	\label{eq4.107}
	 &({\bf e}^{n+1}_{\bf u}-{\bf e}^n_{\bf u},{\bf v}_h ) + k (\nab {\bf e}^{n+1}_{\bf u},\nab {\bf v}_h )  
	+ k (\nabla {e}^{n+1}_r, {\bf v}_h ) = ( [\pmb{\eta}^{n}_{\varepsilon}-\pmb{\eta}_{\varepsilon,h}^{n}]\Delta_{n+1} W, {\bf v}_h),\\
	\label{eq4.108}	
	&\big(\div {\bf e}^{n+1}_{\bf u},q_h\big) + \varepsilon(\nabla e^{n+1}_r, \nabla q_h) =0\, .
	\end{align}
	Now consider \eqref{eq4.107}--\eqref{eq4.108} for $\omega \in \Omega$ fixed, and choose
 $$({\bf v}_h, q_h) = \bigl(\mathbf{Q}_h {\bf e}^{n+1}_{\bf u}(\omega),
 \widetilde{\mathcal R}_h  e^{n+1}_r(\omega)\bigr) \in {\mathbb Y}_{h} \times W_{h}\, ;$$ we then deduce
\begin{align}\label{eq4.107a}
 ({\bf e}^{n+1}_{\bf u} -{\bf e}^{n}_{\bf u}, \mathbf{Q}_h {\bf e}^{n+1}_{\bf u}  ) &+ k (\nab {\bf e}^{n+1}_{\bf u},\nab  \mathbf{Q}_h {\bf e}^{n+1}_{\bf u}  ) -k (e^{n+1}_r,\div \mathbf{Q}_h {\bf e}^{n+1}_{\bf u} )   \\
&= \bigl(  [\pmb{\eta}^{n}_{\varepsilon}-\pmb{\eta}_{\varepsilon,h}^{n}]\Delta_{n+1} W, \mathbf{Q}_h {\bf e}^{n+1}_{\bf u} \bigr)\, . \nonumber
\end{align}
We can adopt the corresponding arguments in (\ref{eq4.18a}) and (\ref{eq4.18b}),
and use Lemma \ref{stabilty-thm} to treat the first two terms in (\ref{eq4.107a}), and also the argument
 around (\ref{eq4.18d}) may easily be adopted to the present setting. 
 But a different treatment is required to deal with the last term on the left-hand side of (\ref{eq4.107a}) 
 because it involves the error in the pressure. We rewrite this term as follows:
\begin{eqnarray*}
&& {\tt II} :=  \Bigl(e^{n+1}_r,\div \bigl[ {\bf e}^{n+1}_{\bf u} + (\mathbf{Q}_h {\bf u}^{n+1}_\varepsilon - {\bf u}^{n+1}_\varepsilon) \bigr]\Bigr) \\
&&\quad =  \bigl( \widetilde{\mathcal R}_h e^{n+1}_r, \div   {\bf e}^{n+1}_{\bf u} \bigr) +
 \bigl(r^{n+1}_{\varepsilon} -  \widetilde{\mathcal R}_h  r^{n+1}_{\varepsilon}, \div   {\bf e}^{n+1}_{\bf u} \bigr) + \bigl( e^{n+1}_r, \div [\mathbf{Q}_h {\bf u}^{n+1}_\varepsilon - {\bf u}^{n+1}_\varepsilon] \bigr) \\
 &&\quad =: {\tt II}_1 + {\tt II}_2 + {\tt II}_3\, .
\end{eqnarray*}
We estimate ${\tt II}_2$ with the help of (\ref{eq4.2b}) and using Lemma \ref{stabilty-thm},
\begin{eqnarray*}
{\mathbb E}\bigl[\vert {\tt II}_2\vert \bigr] \leq Ch^2 {\mathbb E}[\Vert \nabla r^{n+1}_{\varepsilon}\Vert^2] + 
\frac{1}{4} {\mathbb E}\bigl[\Vert \nabla {\bf e}^{n+1}_{\bf u}\Vert^2\bigr]\,.
\end{eqnarray*}
Integrating by parts in ${\tt II}_3$, using (\ref{eq4.100c}) and again Lemma \ref{stabilty-thm} yield
\begin{eqnarray*} \vert {\tt II}_3\vert  &=& \bigl\vert \bigl( \widetilde{\mathcal R}_h  e^{n+1}_r, \div [\mathbf{Q}_h {\bf u}^{n+1}_\varepsilon - {\bf u}^{n+1}_\varepsilon] \bigr) \bigr\vert
+ \bigl\vert \bigl(r^{n+1}_{\varepsilon} -  \widetilde{\mathcal R}_h r^{n+1}_{\varepsilon},
\div [\mathbf{Q}_h {\bf u}^{n+1}_\varepsilon - {\bf u}^{n+1}_\varepsilon] \bigr) \bigr\vert \\
&\leq&  \frac{\varepsilon}{4} \Vert \nabla  \widetilde{\mathcal R}_h e^{n+1}_r\Vert^2 + \frac{C h^4}{\varepsilon}  \Vert \Delta {\bf u}^{n+1}_{\varepsilon}\Vert^2
+ Ch^2 \Vert \nabla r^{n+1}_\varepsilon\Vert \Vert \Delta {\bf u}^{n+1}_\varepsilon\Vert\, .
\end{eqnarray*}
Because of (\ref{eq4.108}), we have 
$${\tt II}_1 = -\varepsilon \bigl(\nabla e^{n+1}_r, 
\nabla [ \widetilde{\mathcal R}_h e^{n+1}_r]\bigr) =
-\varepsilon \Vert \nabla \widetilde{\mathcal R}_h e^{n+1}_r\Vert^2 \, .$$
Putting the above auxiliary estimates together, we obtain that there exists some $h$- and $\varepsilon$-independent constant $C>0$ such that
\begin{align} \label{eq5.125b} 
	\frac{1}{2} {\mathbb E}\bigl[\|\mathbf{Q}_h {\bf e}^{m+1}_{\bf u} \|^2\bigr] &+ \frac{1}{4}\sum_{n=0}^{m} {\mathbb E}\bigl[\|\mathbf{Q}_h ({\bf e}^{n+1}_{\bf u} - {\bf e}^{n}_{\bf u}) \|^2 \bigr] \\ 
	&+ \frac{1}{4} k\sum_{n=0}^{m} {\mathbb E}\big[\|\nab {\bf e}^{n+1}_{\bf u}\|^2 +
	\varepsilon \Vert \nabla \widetilde{\mathcal R}_h  e^{n+1}_r \Vert^2 \big] 
	\leq C \Bigl(h^2 + \frac{h^4}{\varepsilon}\Bigr) \nonumber
\end{align}
for every $0 \leq m \leq N$. The desired estimates \eqref{eq4.13b} then follows from an 
application of the discrete Gronwall inequality. The proof is complete. 
\end{proof}

The last result gives an estimate for the pressure approximation error. Using (\ref{lbb_disk}), equation (\ref{eq4.107}) after taking a summation in $n$, and Lemma \ref{stabilty-thm}, we obtain

  \begin{theorem}\label{thm3.4aaab}
	Let $\{r^n_\varepsilon; 1\leq n\leq N\}$ be the pressure in {\rm Algorithm 3}, and
	$\{r^n_{\varepsilon,h}; 1\leq n\leq N\}$ be the pressure in {\rm Algorithm 4}.   
		There exists a constant $C>0$, 
		such that for all $1 \leq m \leq N$ 
	\begin{align*}
	\Bigl(\mathbb{E}\Bigl[ \|k\sum^m_{n=1}r^n_{\varepsilon}-k\sum^m_{n=1}r^n_{\varepsilon,h}  \|^2 + \|k\sum^m_{n=1}p^n_{\varepsilon}-k\sum^m_{n=1}p^n_{\varepsilon,h}  \|^2\Bigr] \Bigr)^{\frac12}
	\leq C  \,\Bigl(h  + \frac{h^2}{\sqrt{\varepsilon}}\Bigr)\, .
	\end{align*}
\end{theorem}

To sum up the results in this section, we have shown the following error estimates 
for {\em Algorithm 4}.

\begin{theorem}\label{thm5.5a}
Let $({\bf u}, P)$ be the solution of (\ref{eq1.1}) 
and $\{ ({\bf u}^n_{\varepsilon,h}, r_{\varepsilon,h}^n, p_{\varepsilon,h}^n); 1\leq n\leq N \}$ be the solution of {\em Algorithm 4}. There exists a constant $C>0$, 
such that 
	\begin{align*} 
	{\rm (i)}\, \max_{1\leq n\leq N} \Bigl(\mathbb{E}\bigl[\|\bu(t_n)-\bu^n_{\varepsilon,h}\|^2\,\bigl]\Bigr)^{\frac12}
	 &+\Bigl( \mathbb{E}\Bigl[ k\sum_{n=1}^{N}\|\nabla (\bu(t_n) -\bu^n_{\varepsilon,h})\|^2\,\Bigr] \Bigr)^{\frac12} \\
	&\leq C \biggl( k^{\frac12}+  h + \frac{h^2}{\sqrt{\varepsilon}} \biggr)\,, \\ 
	{\rm (ii)}\, \Bigl(\mathbb{E} \bigl[ \bigl\|R(t_m) -k\sum^m_{n=1}r^n_{\varepsilon,h} \bigr\|^2\, \bigr]\Bigr)^{\frac12} 	
	&+ \Bigl( \mathbb{E}\bigl[\bigl\|P(t_m) -k\sum^m_{n=1}p^n_{\varepsilon,h} \bigr\|^2\,\bigr] \Bigr)^{\frac12} \\
 &\leq C \, \biggl(k^{\frac12}+ h + \frac{h^2}{\sqrt{\varepsilon}} \biggr)\, .
	\end{align*}
\end{theorem}

\section{Computational experiments}\label{sec6}\label{sec-6}
We present computational results to validate the theoretical error estimates in Theorems \ref{thm4.5} and \ref{thm5.5a}, and evidence how crucial the
numerical treatment of the pressure part in the noise is to obtain an optimally
convergent mixed method for (\ref{eq1.1}).
Our computations are done using the software packages FreeFem++ \cite{H08} 
and Matlab, and the physical domain of all experiments is taken to be $D=(0,1)^2$, 
i.e., $L=1$.

Specifically, we use Algorithm 2 to compute the solution of the following initial-(Dirichlet) 
boundary value problem:
\begin{subequations} 
	\begin{alignat}{2} \label{e6.1a}
	d{\bf u} &=\bigl[ \Delta {\bf u} -\nabla p + \vf \bigr] dt +{\bf B}({\bf u}) dW(t)  &&\qquad\mbox{ in}\, D_T:=(0,T)\times D,\\
	\div {\bf u} &=0 &&\qquad\mbox{ in}\, D_T,\label{e6.1b}\\
		{\bf u} &=0 &&\qquad\mbox{ on}\, \partial D_T:=(0,T)\times \partial D, \label{e6.1c}\\
	{\bf u}(0)&= {\bf u}_0 &&\qquad\mbox{ in}\, D,\label{e6.1d}
	\end{alignat}
\end{subequations}
and use Algorithm 4 to compute the solution of the pressure stabilization of the above system 
which is obtained by replacing \eqref{e6.1b} by \eqref{e6.2a}--\eqref{e6.2b} below. 
\begin{subequations}\label{e6.2}
\begin{alignat}{2}\label{e6.2a}
{\rm div}\, {\bf u} - \varepsilon \Delta p  &= 0 &&\qquad \mbox{in } D_T,  \\
\partial_{\bf n} p  &= 0 &&\qquad \mbox{on } \partial D_T, \label{e6.2b}
\end{alignat}
\end{subequations}
where $\partial_{\bf n} p$ stands for the normal derivative of $p$.

\textbf{Test 1.}  Let ${\bf u}_0=(0,0), \vf = (1,1)^\top$ and  ${\bf B}(u_1,u_2) = \bigl((u_1^2+1)^{\frac12}, (u_2^2+1)^{\frac12} \bigr)$, which is non-solenoidal. 
We choose $W$ in \eqref{eq1.1} to be a ${\mathbb R}^J$-valued Wiener process, with  increment
\begin{align}\label{used_noise}
W^J(t_{n+1},{\bf x}) - W^J(t_n,{\bf x}) = k \sum\limits_{j=1}^J\sum\limits_{j_1=1}^J \sqrt{\lambda_{j_1j_2}}g_{j_1j_2}({\bf x})\xi_{j_1j_2}^n\,,
\end{align}
where ${\bf x}=(x_1,x_2) \in D, \, \xi_{j_1j_2}^n \sim N(0,1)$, $\lambda_{j_1j_2} = \frac{1}{j_1^2 
	+ j_2^2}$, and
\begin{align}
g_{j_1j_2}({\bf x}) = 2\sin(j_1\pi x_1)\sin(j_2 \pi x_2)\, .
\end{align}
We use the following parameters: $J = 4$ and $T = 1$, and  
take $N_p = 501$ to be the number of realizations in this test. 
 
Let $k_0$ and $k$ denote the fine and regular time step sizes which are used to generate 
the numerical true solution and a computed solution, clearly $k_0<<k$. Moreover, 
$({\bf u}_h^n(\tau), p_h^n(\tau))$ denote the numerical solution at the time step $t_n$ 
using the time step size $\tau$; below, $\tau=k_0$ or $k$. 
For any $1\leq n \leq N$, we use the following numerical integration formulas:
\begin{align*}
&\pmb{\mathcal{E}}_{\bu,0}^n :=\Bigl(\mE\Bigl[\|{\bf u}(t_n) -{\bf u}_h^n(k)\|^2 \Bigr]\Bigr)^{\frac12} \approx
\Bigl(\dfrac{1}{N_p}\sum_{\ell=1}^{N_p}\|{\bf u }_h^n(k_0,\omega_\ell)-{\bf u}_h^n(k,\omega_\ell)\|^2 \Bigr)^{\frac12}\, ,\\
&\pmb{\mathcal{E}}_{\bu,1}^n  :=\Bigl(\mE\Bigl[\|\nabla ( {\bf u}(t_n)  -{\bf u}_h^n(k))\|^2 \Bigr]\Bigr)^{\frac12} \approx
\Bigl(\dfrac{1}{N_p}\sum_{\ell=1}^{N_p}\|\nabla ({\bf u }_h^n(k_0,\omega_\ell)-{\bf u}_h^n(k,\omega_\ell))\|^2 \Bigr)^{\frac12}\, , \\
&{\mathcal{E}}_{p,av}^N :=\Bigl(\mE\Bigl[\Big\|\int_0^T p(s)\, ds -k\sum_{n=1}^{\frac{T}{k}}p_{h}^n(k)\|^2 \Bigr]\Bigr)^{\frac12}\\
&\hskip 0.35in \approx
\Bigl(\dfrac{1}{N_p}\sum_{\ell=1}^{N_p}\Big\|k_0\sum_{n=1}^{\frac{T}{k_0}} p_{h}^n(k_0,\omega_\ell) -k\sum_{n=1}^{\frac{T}{k_i}}p_h^n(k,\omega_\ell)\Bigr\|^2 \Bigr)^{\frac12}\, ,
\end{align*}
and 
\begin{align*}
{\mathcal{E}}_{p,0}^n  :=\Bigl(\mE\Bigl[\|p(t_n)-p_h^n(k)\|^2\Bigr]\Bigr)^{\frac12} \approx
\Bigl(\dfrac{1}{N_p}\sum_{\ell=1}^{N_p}\|p_h^n(k_0,\omega_\ell)-p_h^n(k,\omega_\ell)\|^2\Bigr)^{\frac12}\, .
\end{align*}
The definitions of ${\mathcal{E}}_{r,av}^N$ and ${\mathcal{E}}_{r,0}^N$ are similar.  

We then implement Algorithm 2 and verify the convergence orders of the 
time and spatial discretizations proved in Theorem \ref{thm4.5}.   

To generate a numerical exact solution for computing the orders of convergence, we use $k_0=\frac{1}{600}$ and 
$h_0=\frac{1}{100}$ as fine mesh sizes to compute such a solution. Then, to compute the convergence order of the time discretization for the velocity, we fix $h=\frac{1}{100}$ and then compute the numerical solution with following time mesh sizes: $k=\frac15,\frac{1}{10}, \frac{1}{20}, \frac{1}{40}$. The errors in the $L^2$-norm ($\pmb{\mathcal{E}}_{\bu,0}^n$) and $H^1$-norm 
($\pmb{\mathcal{E}}_{\bu,1}^n$)
 are shown in Table \ref{table1}. The numerical results verify the convergence order $O(k^{\frac12})$ which is
 stated in Theorem \ref{thm4.5}.

\begin{table}[tbhp]
		\begin{center}
			\begin{tabular}{ |c|c|c|c|c|}
				\hline
				\bf $k$ & $\pmb{\mathcal{E}}_{\bu,0}^n$  & \mbox{order} & $\pmb{\mathcal{E}}_{\bu,1}^n$ &  \mbox{order}\\
				\hline 
				$1/5$  & 0.16253 &  & 0.25558&\\
				\hline
				$1/10$  & 0.11521 & 0.496 & 0.18050&0.5018\\
				\hline
				$1/20$  & 0.08145 & 0.5002 & 0.12580&0.5209\\
				\hline
				$1/40$  & 0.05730 & 0.5073 & 0.08758&0.5225\\
				\hline
			\end{tabular}
		\smallskip
			\caption{Algorithm 2: Time discretization errors for the velocity $\{ {\bf u}^n_h\}_n$}.
			\label{table1}
		\end{center}
\end{table}

Tables \ref{table2} and \ref{table22} display respectively  the $L^2$-norm errors $({\mathcal{E}}_{\alpha,av}^N)$ 
and $({\mathcal{E}}_{\alpha,0}^N)$ ($\alpha=r$ and $p$)  of the time-averaged pressure approximations using time mesh sizes: 
$k=\frac15,\frac{1}{10}, \frac{1}{20}, \frac{1}{40}$. The numerical results
indicate the convergence rate $O(k^{\frac12})$ which was predicted in Theorem \ref{thm4.5}. We also present the 
standard $L^2$-norm errors ${\mathcal{E}}_{r,0}^N $ and ${\mathcal{E}}_{p,0}^N$ in 
Table \ref{table2} and \ref{table22} respectively for comparison purposes, for which we observe a 
significantly slower rate. It should be noted that our convergence theory does not 
conclude such a convergence behavior.

\begin{table}[tbhp]
		\begin{center}
			\begin{tabular}{ |c|c|c|c|c|}
				\hline
				\bf $k$ & ${\mathcal{E}}_{r,av}^N$  & \mbox{order} & ${\mathcal{E}}_{r,0}^N $ & \mbox{order}\\
				\hline 
				$1/5$  & 0.06352 &  & 0.08013&\\
				\hline
				$1/10$  & 0.04486 & 0.5019 & 0.06231&0.3629\\
				\hline
				$1/20$  & 0.03161 & 0.5049 & 0.04842&0.3639\\
				\hline
				$1/40$  & 0.02219 & 0.5102 & 0.03734&0.3745\\
				\hline
			\end{tabular}
		\smallskip 
			\caption{Algorithm 2: Time discretization errors for the   pressure  $\{r^n_h\}_n$.}
			\label{table2}
		\end{center}
\end{table}

\begin{table}[tbhp]
		\begin{center}
			\begin{tabular}{ |c|c|c|c|c|}
				\hline
			 \bf $k$ & ${\mathcal{E}}_{p,av}^N$  & \mbox{order} & ${\mathcal{E}}_{p,0}^N $ & \mbox{order}\\
				\hline 
				$1/5$  & 0.00217 &  & 0.0967&\\
				\hline
				$1/10$  & 0.00154 & 0.4947 & 0.0722&0.3211\\
				\hline
				$1/20$  & 0.00109 & 0.4986 & 0.0579&0.3184\\
				\hline
				$1/40$  & 0.00077 & 0.5014 & 0.0461&0.3288\\
				\hline
			\end{tabular}
			\smallskip 
			\caption{Algorithm 2: Time discretization errors for the pressure approximation $\{p^n_h\}_n$.}
			\label{table22}
		\end{center}
\end{table}

To verify the convergence rate $O(h)$ for the velocity approximation, we fix $k = \frac{1}{200}$ and use different 
spatial mesh sizes $h = \frac15,\frac{1}{10}, \frac{1}{20}, \frac{1}{40}$ to compute the errors 
$\pmb{\mathcal{E}}_{\bu,0}^n$  and $\pmb{\mathcal{E}}_{\bu,1}^n$.
Table \ref{table3} contains the computational results which verify first order convergence rate for  
both  as stated in Theorem \ref{thm4.5}.  

\begin{table}[tbhp]
		\begin{center}
			\begin{tabular}{ |c|c|c|c|c|}
				\hline
		 		\bf $h$ & $\pmb{\mathcal{E}}_{\bu,0}^n$  & \mbox{order} & $\pmb{\mathcal{E}}_{\bu,1}^n$ &  \mbox{order}\\
				\hline 
				$1/5$  & 0.07981 &  & 0.50832&\\
				\hline
				$1/10$  & 0.04034 & 0.9844 & 0.25315&1.0057\\
				\hline
				$1/20$  & 0.02016 & 1.0007 & 0.12662&0.9995\\
				\hline
				$1/40$  & 0.01007 & 1.0014 & 0.06322&1.0021\\
				\hline
			\end{tabular}
			\caption{Algorithm 2: Spatial discretization errors for the velocity approximation $\{ {\bf u}^n_h\}_n$.}
			\label{table3}
		\end{center}
\end{table}

To verify the convergence rate for the pressure approximation, we fix $k=\frac{1}{200}$ and use different spatial 
 mesh sizes: $h = \frac15,\frac{1}{10}, \frac{1}{20}, \frac{1}{40}$. Tables \ref{table4} and  \ref{table44} display the error ${\mathcal{E}}_{p,av}^N$  of 
  the pressure approximation. It is evident that ${\mathcal{E}}_{p,av}^N$ converges linearly in $h$ as stated in Theorem 
   \ref{thm4.5}.  For comparison purposes, we also compute the error ${\mathcal{E}}_{p,0}^N $  and include it in Table  \ref{table4} and \ref{table44}.  
The numerical results suggest that the error ${\mathcal{E}}_{p,0}^N $  converges with a slower rate.

\begin{table}[tbhp]
		\begin{center}
			\begin{tabular}{ |c|c|c|c|c|}
				\hline
			 	\bf $h$ & ${\mathcal{E}}_{p,av}^N$  & \mbox{order} & ${\mathcal{E}}_{p,0}^N $ & \mbox{order}\\
				\hline 
				$1/5$  & 0.04289 &  & 0.30972&\\
				\hline
				$1/10$  & 0.02145 & 0.9997 & 0.23572&0.3939\\
				\hline
				$1/20$  & 0.01071 & 1.0022 & 0.17977&0.3901\\
				\hline
				$1/40$  & 0.00534 & 1.0038 & 0.13620&0.3972\\
				\hline
			\end{tabular}
			\caption{Algorithm 2: Spatial discretization errors for the pressure approximation $\{r^n_h\}_n$.}
			\label{table4}
		\end{center}
\end{table}

\begin{table}[tbhp]
		\begin{center}
			\begin{tabular}{ |c|c|c|c|c|}
				\hline
				\bf $h$ & ${\mathcal{E}}_{p,av}^N$  & \mbox{order} & ${\mathcal{E}}_{p,0}^N $ & \mbox{order}\\
				\hline 
				$1/5$  & 0.127620 &  & 0.44524&\\
				\hline
				$1/10$  & 0.068161 & 0.9048 & 0.36504&0.2865\\
				\hline
				$1/20$  & 0.036068 & 0.9182 & 0.29707&0.2973\\
				\hline
				$1/40$  & 0.019262 & 0.9049 & 0.24189&0.2965\\
				\hline
			\end{tabular}
			\caption{Algorithm 2: Spatial discretization errors for the pressure $\{p^n_h\}_n$.}
			\label{table44}
		\end{center}
\end{table}

\medskip
\textbf{Test 2.}  In this test, we  use Algorithm 2 to solve the 
driven cavity problem with stochastic forcing, which is described by system \eqref{eq1.1a}--\eqref{eq1.1b} with the 
following non-homogeneous boundary condition: 
\[
u(x_1, x_2) = \begin{cases}
(1,0),  &\quad \mbox{$x_2=1, \, 0<x_1<1$},   \\
0,  &\quad\mbox{otherwise}.
\end{cases} 
\]
Let $W$ be the  Wiener process as in (\ref{used_noise}), and ${\bf B} \equiv (1,1)^\top$, {\em i.e.}, the 
noise is additive.  We use the following parameters in the test: $T =1$,  $h=\frac{1}{20}$, $k=0.01$, and the number of the realizations is $N_p = 1001$. 

Figure \ref{fig5} displays the expected values of the computed stochastic velocity ${\bf u}^N_h$ 
and pressure $p^N_h$; expectedly, they behave similarly as their deterministic counterparts do.  
On the other hand, individual realizations of the computed stochastic velocity ${\bf u}^N_h$ 
and pressure $p^N_h$ given in Figures \ref{fig6}--\ref{fig8} show quite different behaviors from their deterministic counterparts.

\begin{figure}[thb]
	\begin{center}
		\includegraphics[scale=0.2]{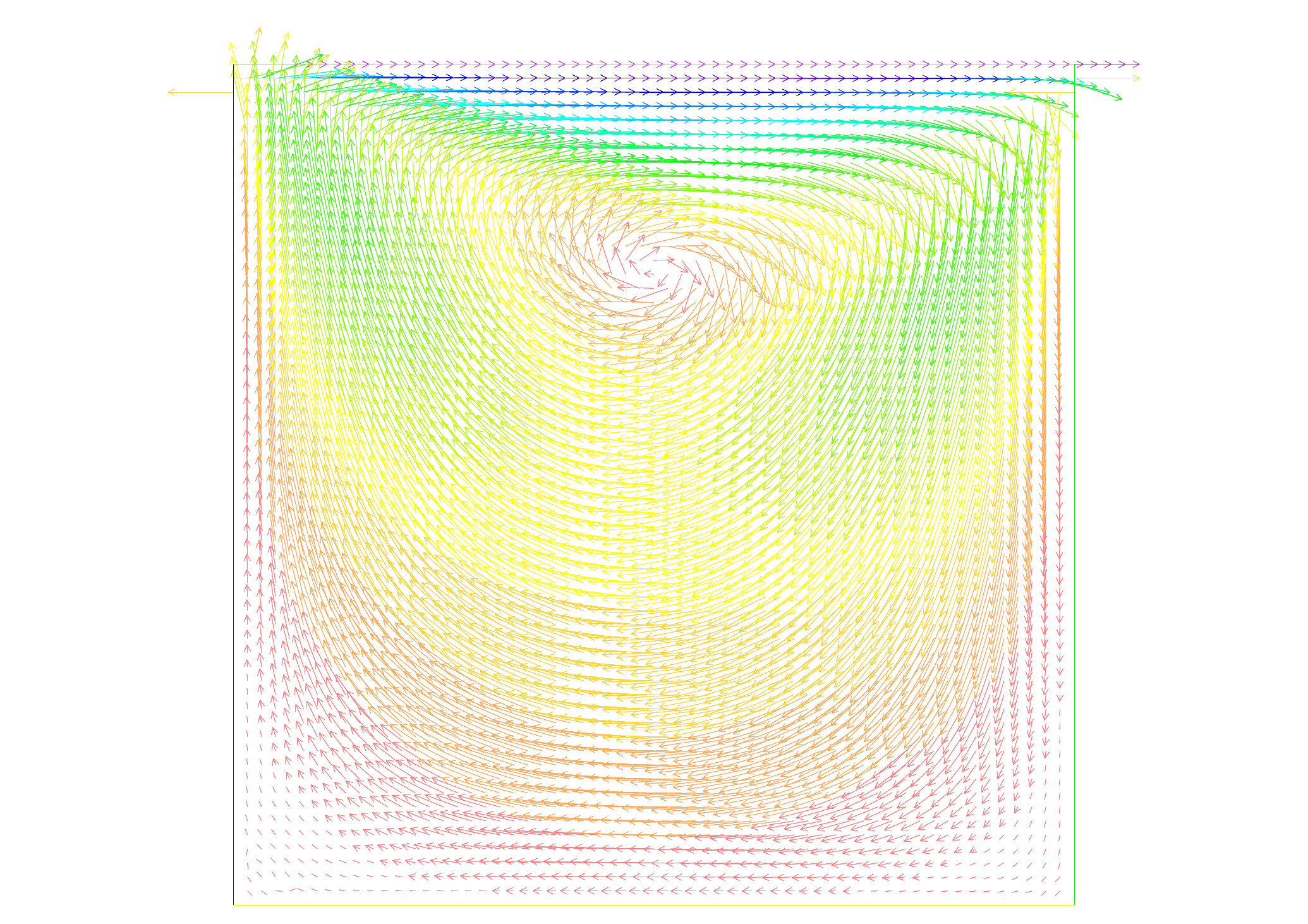}
		\includegraphics[scale=0.2]{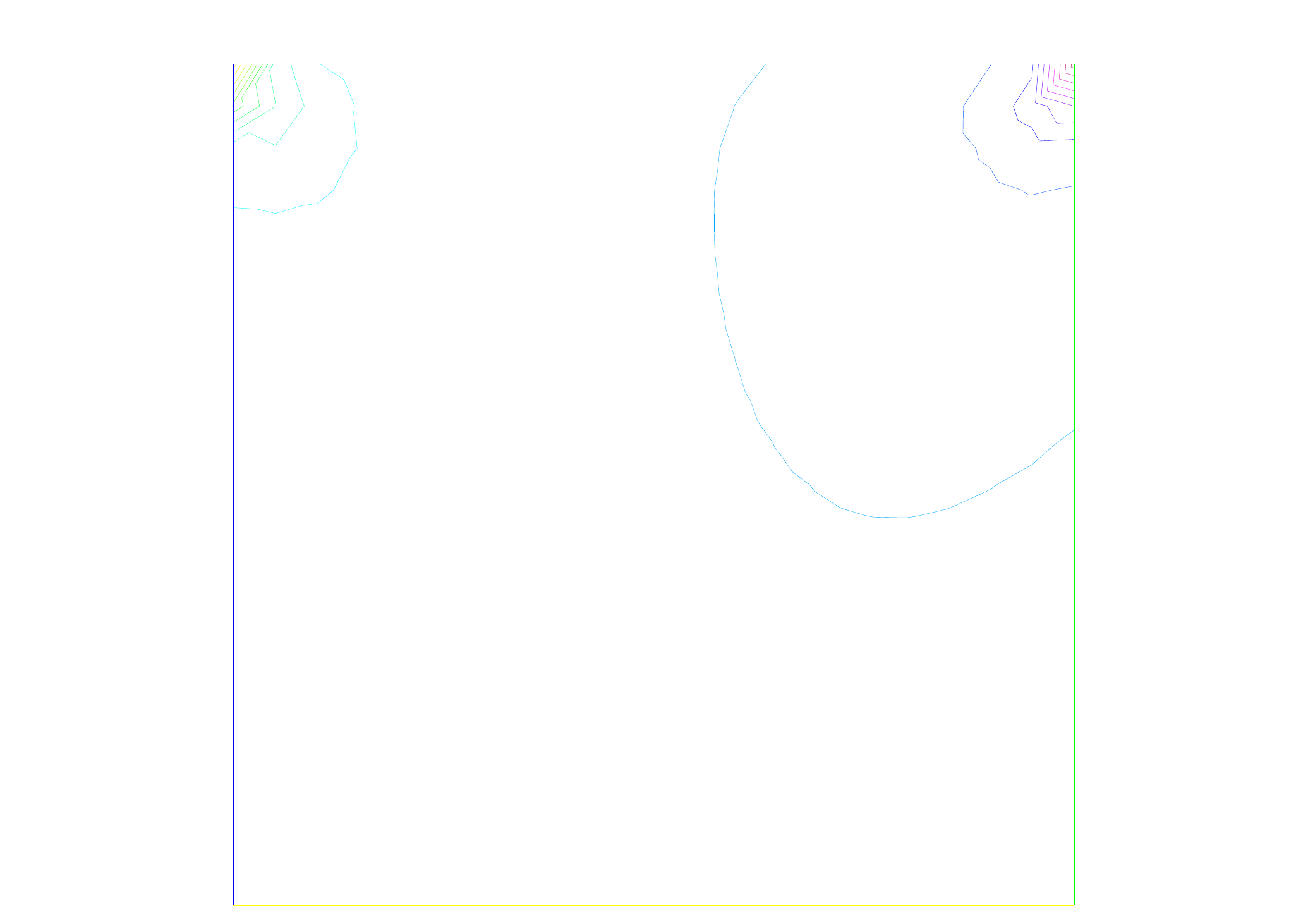}
		\includegraphics[scale=0.2]{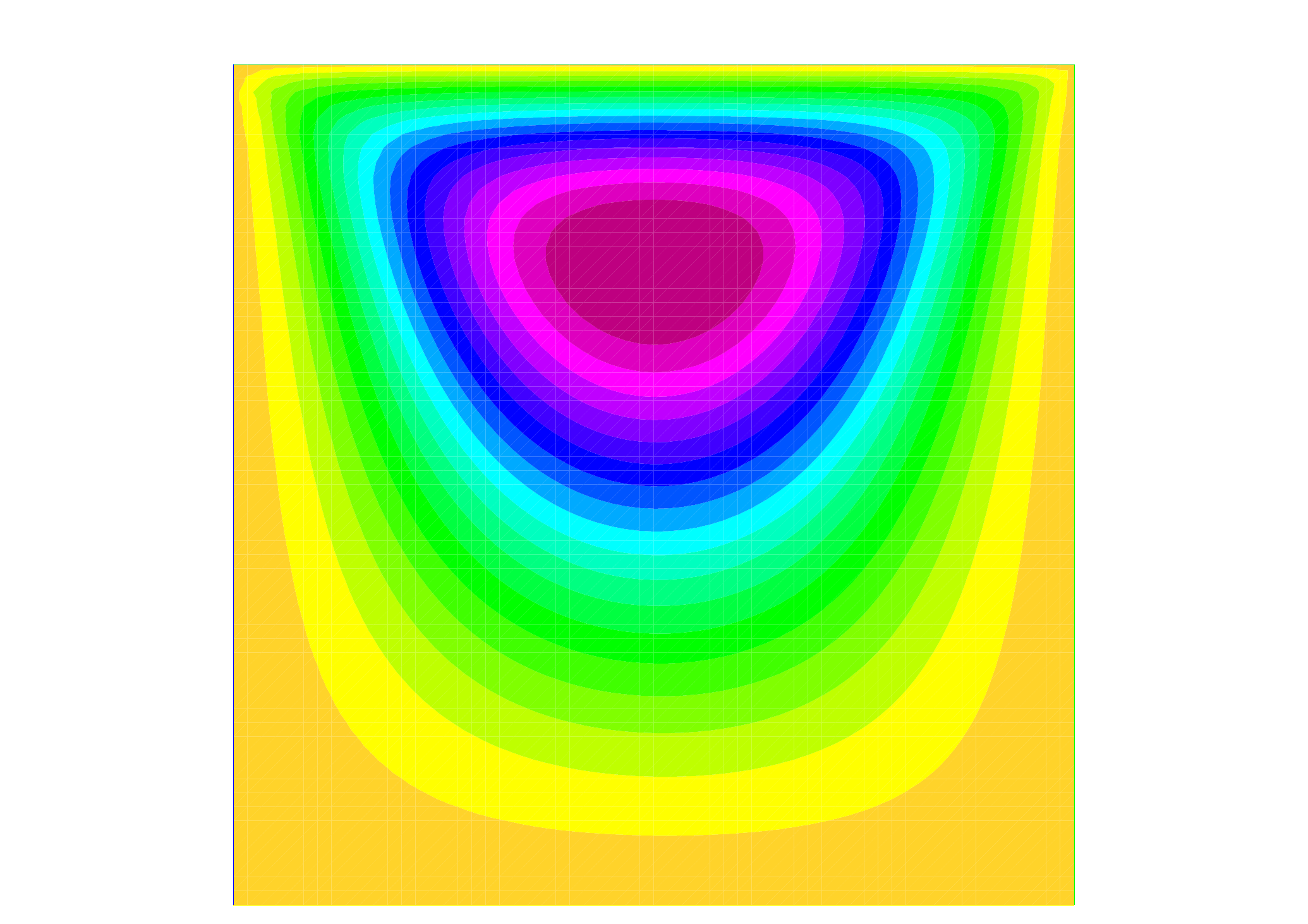}
		\caption{Test 2: (a) The expected value of $\{{\bf u}^N_h\}_n$. (b) Level-lines of the expected value of $\{p^N_h\}_n$. (c) The streamlines of the expected value of $\{{\bf u}^N_h\}_n$.}
		\label{fig5}
	\end{center}
\end{figure}

\begin{figure}[thb]
	\begin{center}
		\includegraphics[scale=0.2]{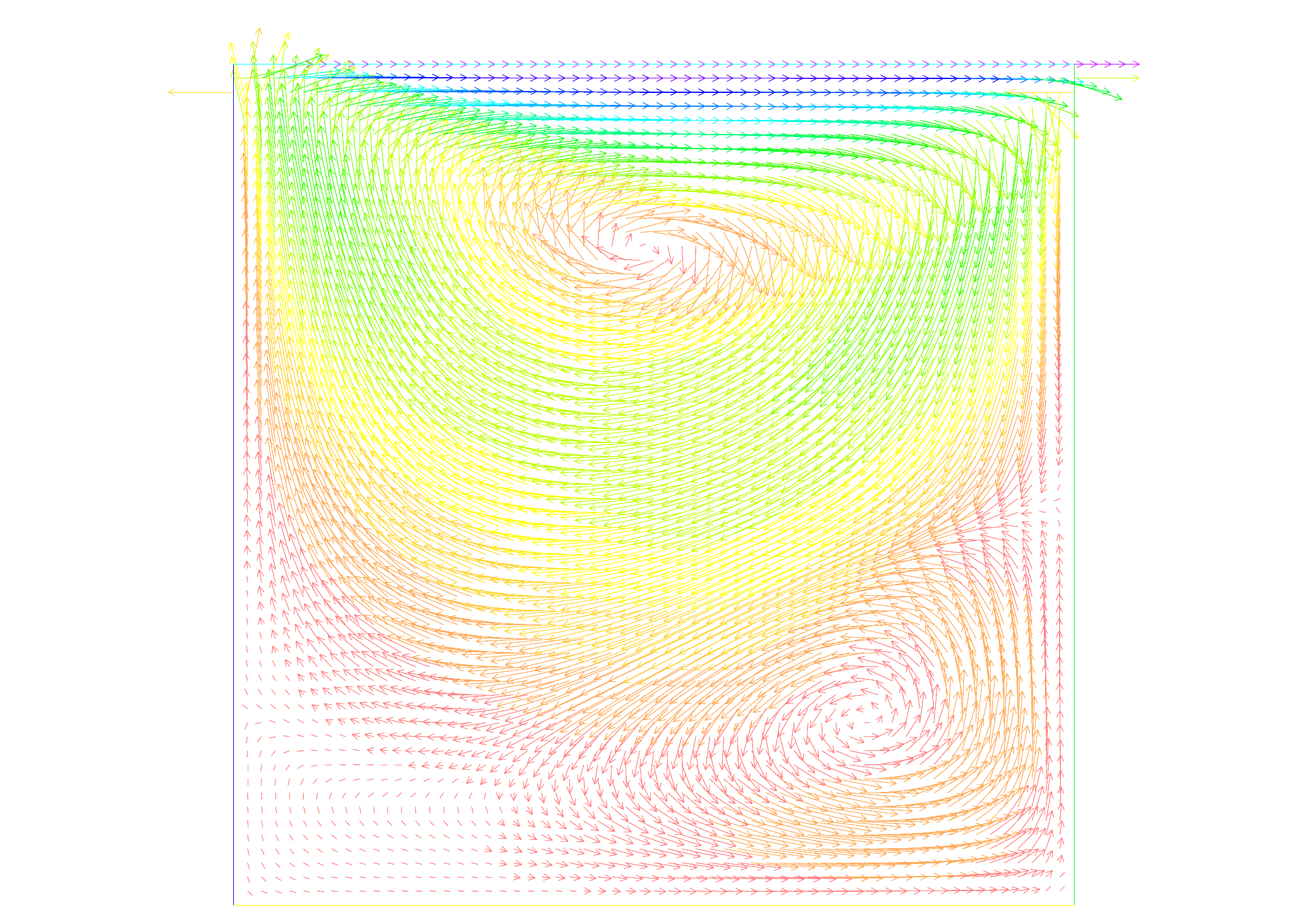}
		\includegraphics[scale=0.2]{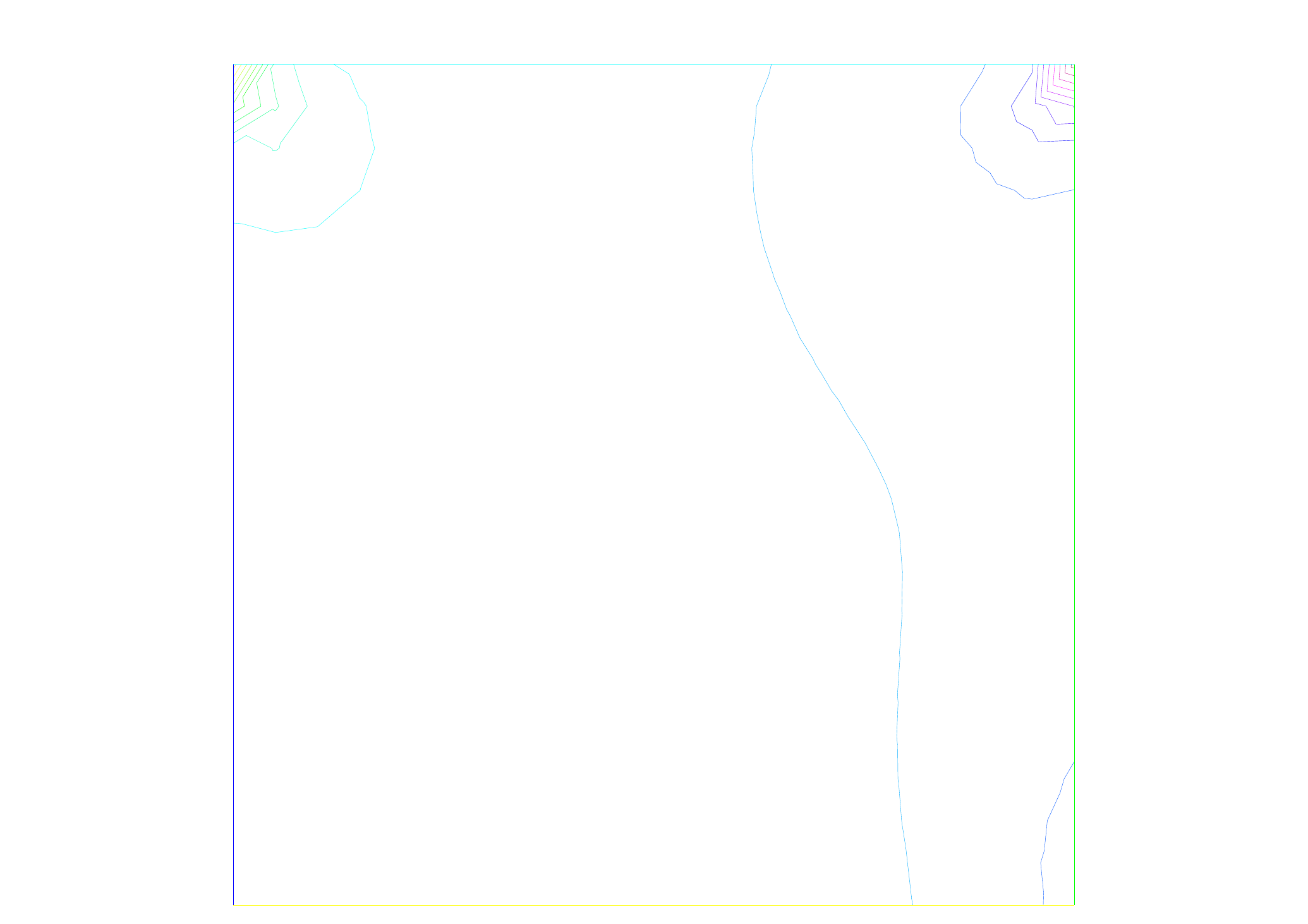}
		\includegraphics[scale=0.2]{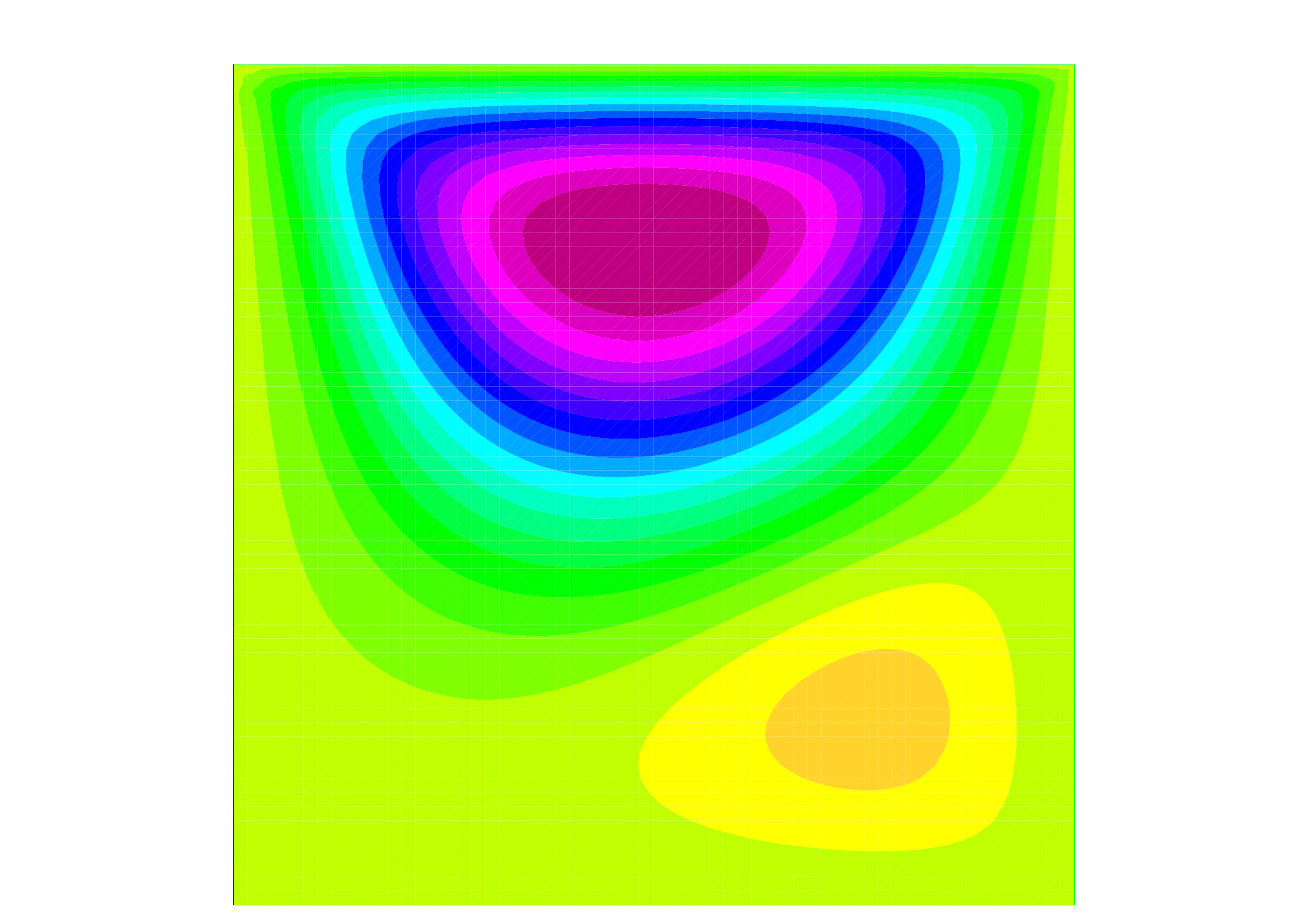}
		\caption{First realization of (a) the velocity $\{{\bf u}^N_h\}_n$;  (b) the pressure $\{p^N_h\}_n$; (c) the streamline of $\{{\bf u}^N_h\}_n$.}
		\label{fig6}
	\end{center}
\end{figure}

\begin{figure}[thb]
	\begin{center}
		\includegraphics[scale=0.2]{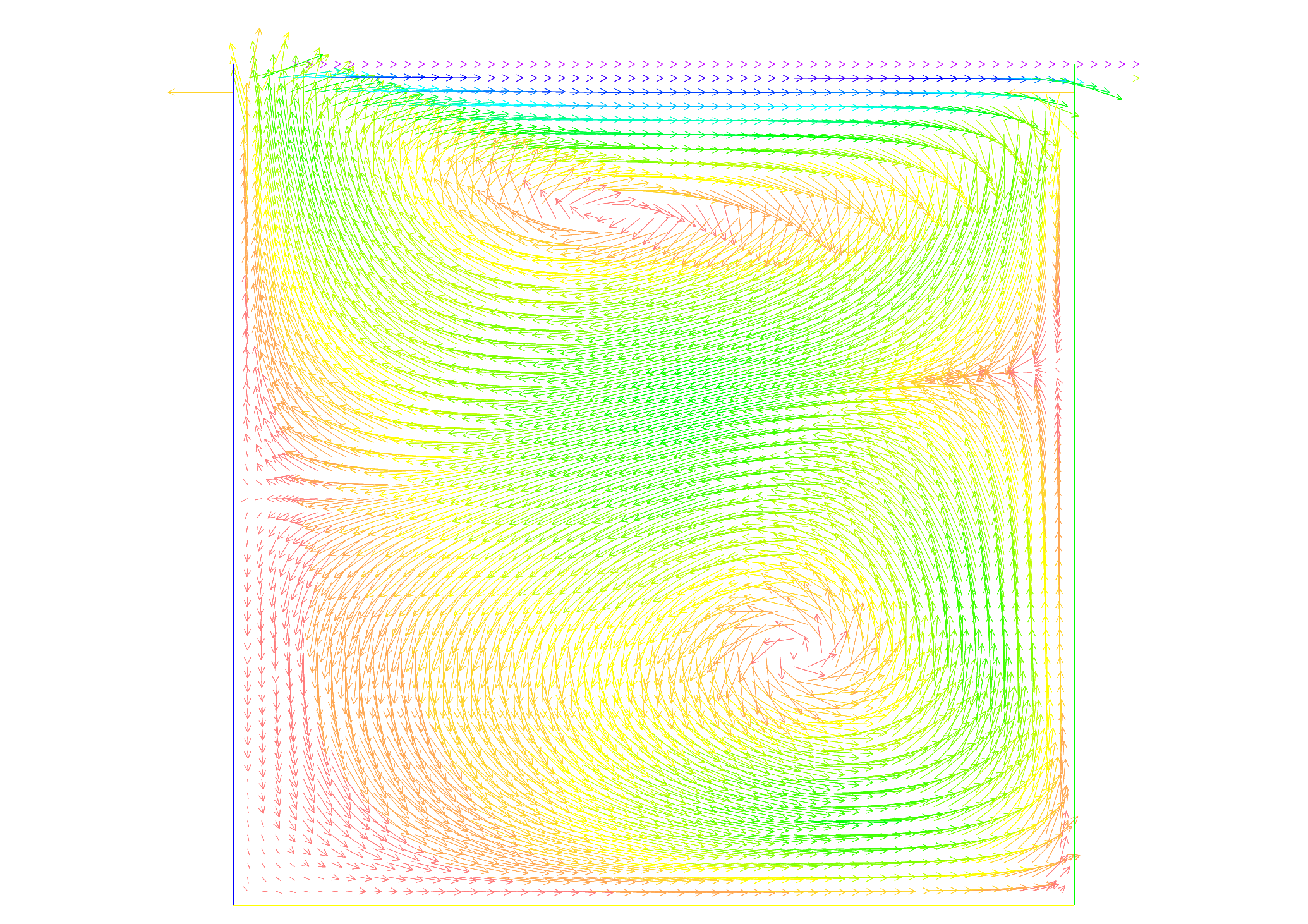}
		\includegraphics[scale=0.2]{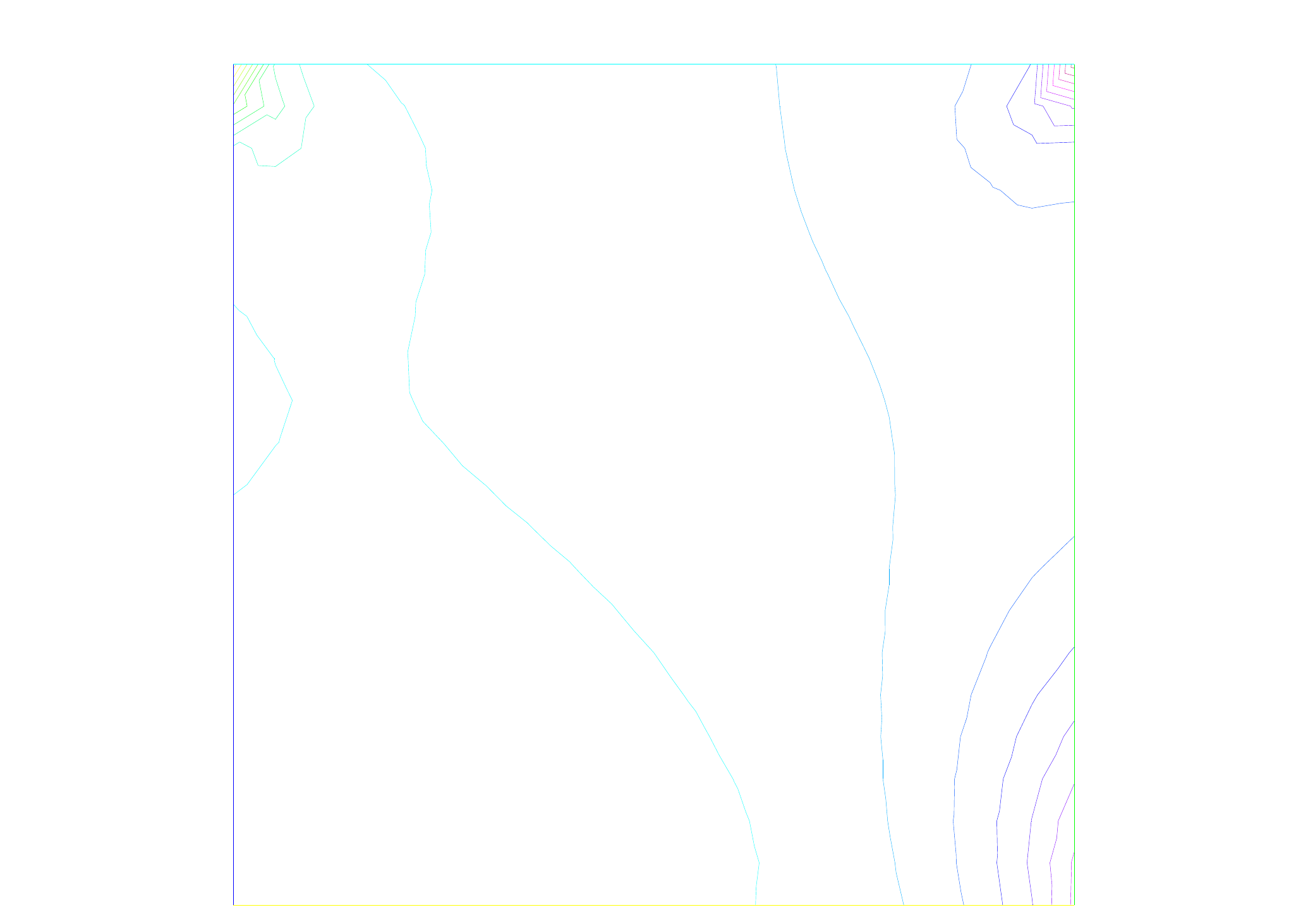}
		\includegraphics[scale=0.2]{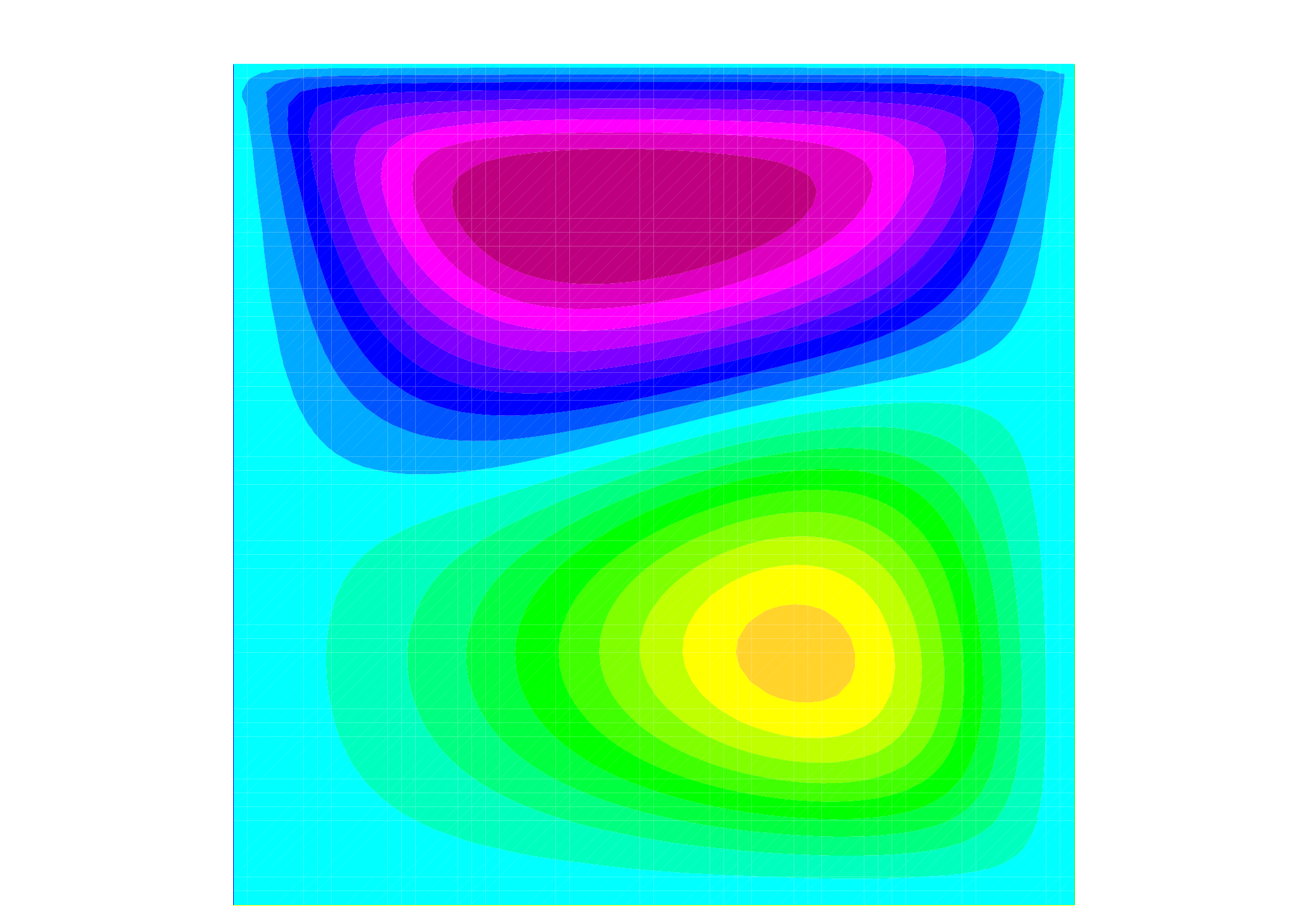}
		\caption{Second realization of (a) the velocity $\{{\bf u}^N_h\}_n$; (b) the pressure $\{p^N_h\}_n$; (c) the streamline of $\{{\bf u}^N_h\}_n$.}
		\label{fig7}
	\end{center}
\end{figure}

\begin{figure}[thb]
	\begin{center}
		\includegraphics[scale=0.2]{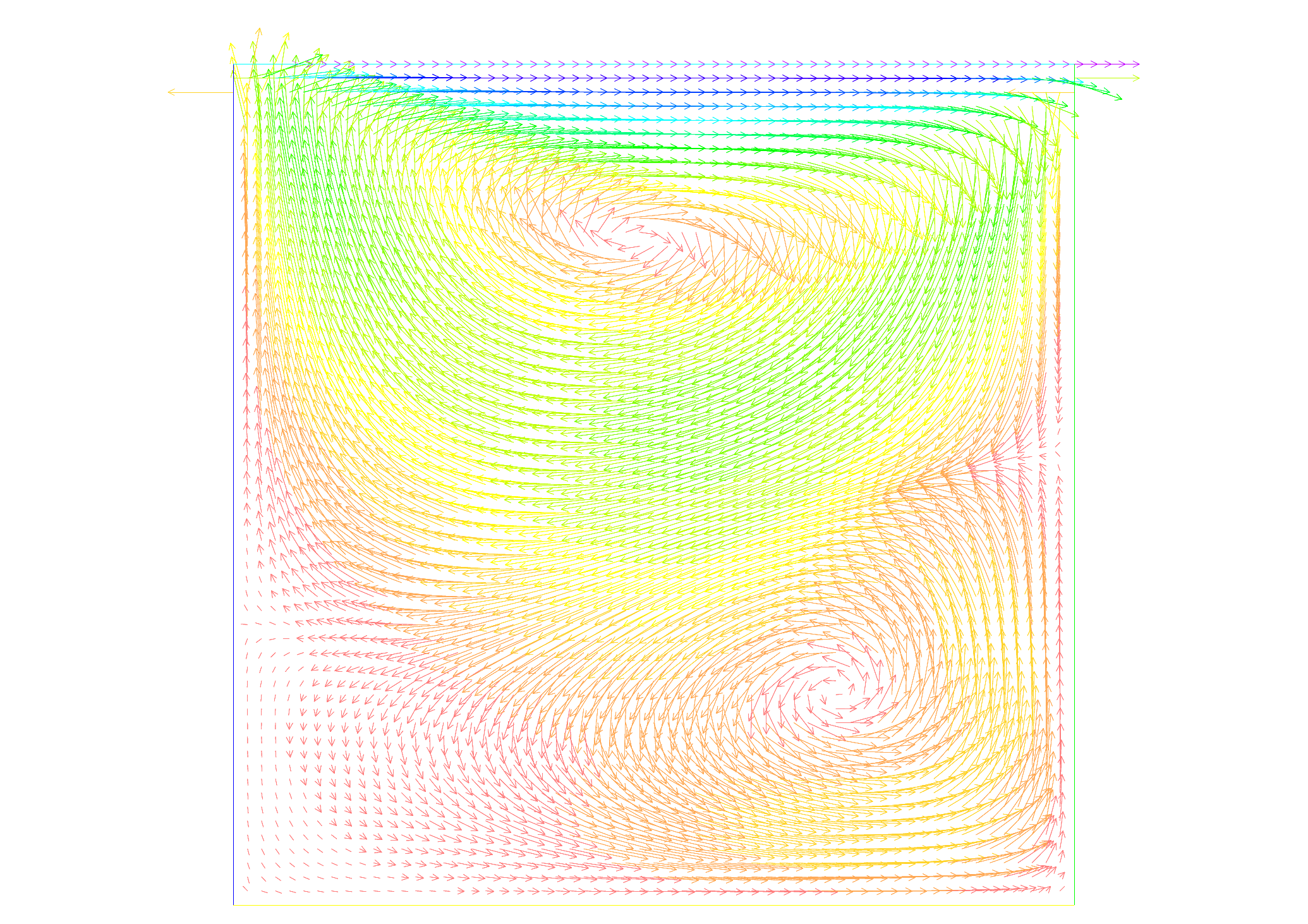}
		\includegraphics[scale=0.2]{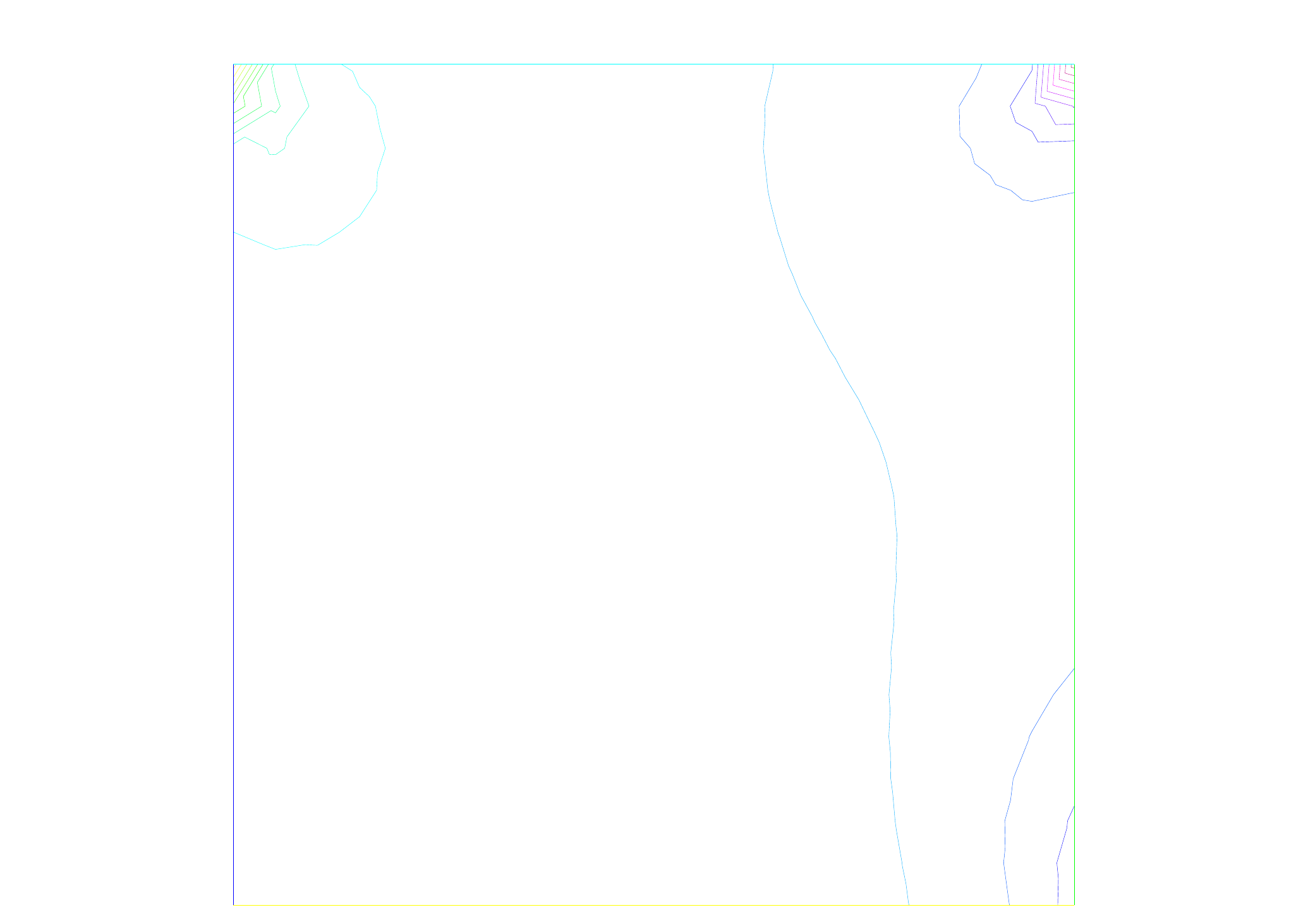}
		\includegraphics[scale=0.2]{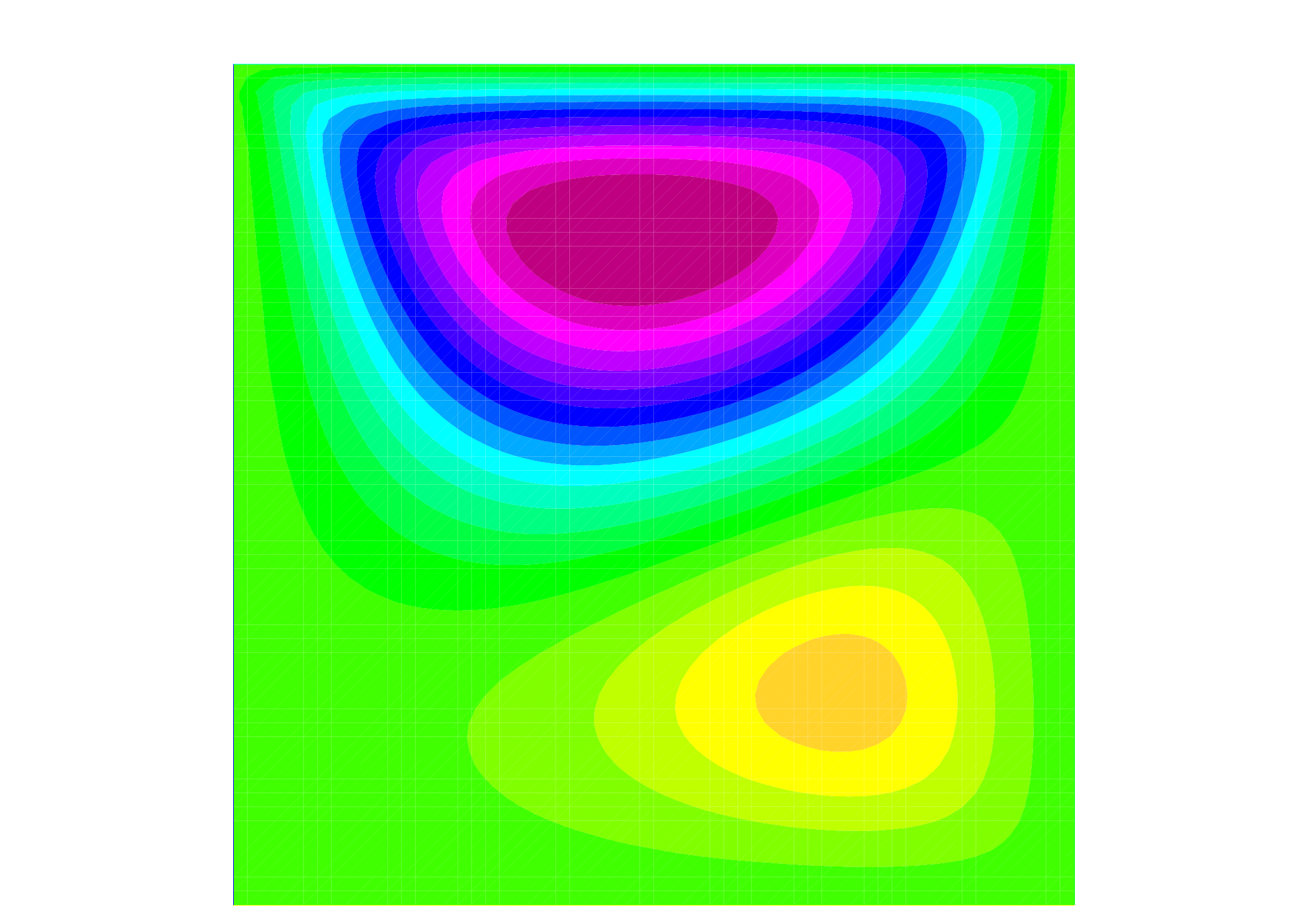}
		\caption{Third realization of (a) the velocity $\{{\bf u}^N_h\}_n$; (b) the pressure $\{p^N_h\}_n$; (c) the streamline of $\{{\bf u}^N_h\}_n$.}
		\label{fig8}
	\end{center}	
\end{figure}

\medskip
{
	\textbf{Test 3. } In this test, we study the stabilization method 
	in section \ref{sec-5}. Specifically, we implement Algorithm 4 
	with the same function ${\bf B}$ as in Test 1, and $\{W(t); 0\leq t\leq T\}$ is chosen as an $\mathbb{R}$-valued Wiener process.  
	We also add a constant forcing term ${\bf f} \equiv (1,1)^\top$ to \eqref{eq1.1a} in order to
	construct an exact solution to system \eqref{eq1.1}. We also take 
	${\bf u}_0 = (0,0)$, $T = 1$, 
	the number of realizations $N_p = 800$, and the minimum time step $k_0 = \frac{1}{4096}$. 
	The computations are done on a uniform mesh of $D$ with the mesh size $h = \frac{1}{100}$.
	 
	In order to verify the optimal convergence rate $O(h)$ of Theorem \ref{error-thm-h}, we 
	fix $k = \frac{1}{256}$ and $\varepsilon=h^2$, and then compute the numerical solutions for different 
	values of $h$. 
	The standard $L^2$-errors  $\pmb{\mathcal{E}}_{\bu,0}^N$ and ${\mathcal{E}}_{p,0}^N $ for the velocity and pressure approximations 
	are presented in Table \ref{table_stable_helm}. The numerical results verify the first 
	order convergence rate for the spatial approximation of the velocity as stated in  
	Theorem \ref{error-thm-h}.  

	\begin{table}[htbp]
			\begin{center}
				\begin{tabular}{ |c|c|c|c|c|}
					\hline
					\bf $h$ &  $\pmb{\mathcal{E}}_{\bu,0}^N$  & order & ${\mathcal{E}}_{p,0}^N $ & order\\
					\hline 
					$1/5$  &  0.018392 &  & 0.147406&\\
					\hline
					$1/10$  & 0.009083 & 1.0178 & 0.092913&0.6658\\
					\hline
					$1/20$  & 0.004095 & 1.1493 & 0.052611&0.8205\\
					\hline
					$1/40$  & 0.002279 & 0.8454 & 0.044723&0.2344\\
					\hline
				\end{tabular}
				\caption{Algorithm 4: Spatial discretization errors for the velocity $\{{\bf u}^n_{\varepsilon,h}\}_n$ and pressure $\{p^n_{\varepsilon,h}\}_n$.}
				\label{table_stable_helm}
			\end{center}
	\end{table}

	For comparison purposes, we also implement the `standard' stabilization method, which is based on 
	\eqref{stabilize} instead of \eqref{eq_new_intro_ref_2a}--\eqref{eq_new_intro_ref_2b}, 
	with the same noise and parameters as above. Table \ref{table_stable_general} displays the 
	$L^2$-errors $\pmb{\mathcal{E}}_{\bu,0}^N$  and ${\mathcal{E}}_{p,0}^N $  of the velocity and pressure approximations. The numerical 
	results indicate that the velocity approximation is also convergent but at a slower rate. This 
	confirms the advantages of the proposed {\em Helmholtz decomposition enhanced} stabilization method 
	(Algorithm 4) over the `standard' stabilization method.
}
	
	\begin{table}[tbhp]
			\begin{center}
				\begin{tabular}{ |c|c|c|c|c|}
					\hline
					\bf $h$ & $\pmb{\mathcal{E}}_{\bu,0}^N$ & order & ${\mathcal{E}}_{p,0}^N $ & order\\
					\hline 
					$1/5$  & 0.037658 &  & 0.735843&\\
					\hline
					$1/10$  & 0.025586 & 0.5576 & 0.888352&-0.2717\\
					\hline
					$1/20$  & 0.019342 & 0.4036 & 0.579818&0.6155\\
					\hline
					$1/40$  & 0.011412 & 0.7611 & 0.442691&0.3893\\
					\hline
				\end{tabular}
				\caption{Standard stabilization method: Spatial discretization errors for the velocity $\{{\bf u}^n_{\varepsilon,h}\}_n$ and pressure $\{p^n_{\varepsilon,h}\}_n$.}
				\label{table_stable_general}
			\end{center}
	\end{table}

\medskip
{\bf Acknowledgment.} After this paper was finished, we were 
brought to attention of the reference \cite{Breit} by Professor D. Breit.  We would like 
to thank him for pointing out the reference and for his explanation, and remark that 
the finite element method proposed in \cite{Breit} is essentially equivalent to 
Algorithm 2 of this paper although they are different algorithmically.


\end{document}